\documentclass[10pt,reqno]{amsart}
\usepackage[colorlinks=true,
pdfstartview=FitV, linkcolor=cyan, citecolor=magenta,
urlcolor=]{hyperref}
\usepackage{amsmath,amsfonts,latexsym,amssymb}
\usepackage{amssymb,amsfonts, amsmath}
\usepackage{mathrsfs}
\usepackage[latin1]{inputenc}
\usepackage[T1]{fontenc}
\usepackage{cleveref}
\usepackage{ae,aecompl}
\usepackage{braket}
\usepackage{comment}
\usepackage{amssymb,amsfonts, amsmath}
 \usepackage{mathtools} 
\usepackage[all,arc]{xy}
\usepackage{enumerate}
\usepackage{enumitem}
\usepackage{mathrsfs}
\usepackage{mathabx}
\usepackage{color}
\usepackage{graphicx}
\usepackage{subfig}
\usepackage{verbatim} 
\usepackage{amssymb,amsfonts, amsmath}
\usepackage[all,arc]{xy}
\usepackage{enumerate}
\usepackage{mathrsfs}
\usepackage{todonotes}
\usepackage{hyperref}
\usepackage{xstring}
\usepackage{amsmath}
\usepackage[makeroom]{cancel}
\usepackage{amsmath}
\usepackage[title]{appendix}

\newtheorem{theorem}{Theorem}[section]
\newtheorem{lemma}[theorem]{Lemma}
\newtheorem{proposition}[theorem]{Proposition}
\newtheorem{corollary}[theorem]{Corollary}

\theoremstyle{definition}
\newtheorem{definition}[theorem]{Definition}
\newtheorem{remark}[theorem]{Remark}

\newcommand{\ms}{\mathsf}

\newcommand{\mc}{\mathcal}

\newcommand{\Aut}{\mathrm{Aut}}

\newcommand{\Real}{\mathbb R}
\newcommand{\Sph}{\mathsf{S}^1}
\renewcommand{\P}{\mathbb{P}}
\newcommand{\Z}{\mathbb Z}

\newcommand{\K}{\mathbb{K}}
\newcommand{\C}{\mathbb C}

\newcommand{\GL}{\mathsf{GL}}
\newcommand{\PGL}{\mathsf{PGL}}

\newcommand{\SL}{\mathsf{SL}}

\newcommand{\Sym}{\textup{Sym}}
\newcommand{\Gr}{\mathsf{Gr}}

\newcommand{\eps}{\varepsilon}

\newcommand{\dgr}{d_{\mathsf{Gr}}}
\newcommand{\dpr}{d_{\mathbb{P}}}
\newcommand{\hil}{\mathsf{d}}
\newcommand{\dist}[1]{\textup{dist}\!\left(#1\right)}
\newcommand{\identity}{\mathrm{id}}
\newcommand{\minus}{\smallsetminus}
\newcommand{\st}{:}

\makeatletter
\renewcommand*\env@matrix[1][\arraystretch]{%
\edef\arraystretch{#1}%
\hskip -\arraycolsep
\let\@ifnextchar\new@ifnextchar
\array{*\c@MaxMatrixCols c}}
\makeatother
\date{\today}

\title{Singular value gap estimates for free products of semigroups}
\date{\today}

\author{Konstantinos Tsouvalas}
\address{Max Planck Institute for
  Mathematics in the Sciences, Inselstrasse 22, 04103 Leipzig, Germany}
\email{konstantinos.tsouvalas@mis.mpg.de}

\author{Theodore Weisman}
\address{Department of Mathematics, University of Michigan, Ann Arbor,
  MI 48109, USA
}
\email{tjwei@umich.edu}

\address{}

\begin{document}

\begin{abstract} We establish lower estimates for singular value gaps
  of free products of $1$-divergent semigroups
  $\Gamma_1,\Gamma_2\subset \mathsf{GL}_d(\mathbb{K})$ which are in
  ping-pong position. As an application, we prove that if $\Gamma_1$
  and $\Gamma_2$ are quasi-isometrically embedded subgroups in ping
  pong position, then the group they generate
  $\langle \Gamma_1,\Gamma_2\rangle$ is also quasi-isometrically
  embedded. In addition, we establish that the class of linear
  finitely generated groups, admitting a faithful linear
  representation over $\mathbb{R}$ which is a quasi-isometric
  embedding, is closed under free products.
\end{abstract}

\frenchspacing

\maketitle

\tableofcontents

\section{Introduction}

When a group $G$ acts on a space $X$, combination theorems relying on
``ping-pong'' inclusions of open subsets of $X$ give a robust means of
producing examples of subgroups of $G$ satisfying some desired
property. A typical setup for such a combination theorem is to
consider a pair of sub(semi)groups $\Gamma_1, \Gamma_2$ of the ambient
group $G$, and some configuration of sets in $X$ mapped into each
other by nontrivial elements of $\Gamma_1, \Gamma_2$. One then wishes
to show that the (semi)group $\Gamma$ generated by
$\Gamma_1, \Gamma_2$:
\begin{enumerate}[label=(\alph*)]
\item\label{item:pingpong_free_product} is naturally isomorphic to the
  free product $\Gamma_1 \ast \Gamma_2$, and
\item\label{item:pingpong_metatheorem} also satisfies some property
  shared by both $\Gamma_1$ and $\Gamma_2$.
\end{enumerate}
A simple application of the ping-pong lemma often suffices to prove
\ref{item:pingpong_free_product}, but proving that
\ref{item:pingpong_metatheorem} holds depends on the precise property
in question.

Theorems of this type date back to Klein's original 19th-century
combination theorems for discrete groups of isometries of hyperbolic
space \cite{Klein}, later developed further by Maskit \cite{Maskit88}
and carried through in numerous different contexts by other authors,
see e.g. \cite{LOW, MP, MPS, TW}.

In the present paper, we are concerned with proving ping-pong type
combination theorems for \emph{discrete subgroups of higher-rank
  semisimple Lie groups.} We focus on the case where
$\Gamma_1, \Gamma_2$ are sub(semi)groups of $\GL_d(\K)$, where
$\K = \Real$ or $\C$.

Our main results (\Cref{thm:pingpong_word_bound} and
\Cref{discrete-freeprod-1} below) give estimates for certain
\emph{singular value gaps} of elements in the semigroup
$\Gamma = \langle \Gamma_1, \Gamma_2 \rangle$, in terms of singular
value gaps of elements in $\Gamma_1$ and $\Gamma_2$. In particular,
these estimates can be used to show that singular value gap growth
properties of $\Gamma_1$ and $\Gamma_2$ can be passed to the
combination $\Gamma$, and we prove several corollaries of this type
below.

Our theorems are strong enough to recover combination theorems
regarding free products of \emph{Anosov subgroups} proved by
Dey-Kapovich \cite{Dey-Kapovich} and Danciger-Gu\'eritaud-Kassel
\cite{DGK24} (see also \cite[Thm. 1.3]{DK24}), but they also apply to
groups outside of this context. \Cref{thm:pingpong_word_bound} applies
to any pair of \emph{$1$-divergent semigroups} (a class which in
particular includes all subgroups of $1$-Anosov and relatively
$1$-Anosov groups), and \Cref{discrete-freeprod-1} can be applied to
essentially arbitrary discrete subgroups of semisimple Lie groups.

\subsection{Singular value gap estimates for $1$-divergent semigroups}

Before stating our first main theorem, we set up some notation and
terminology. For a matrix $g\in \mathsf{GL}_d(\mathbb{K})$ we denote
by $\sigma_1(g)\geq \cdots \geq \sigma_d(g)$ (resp.
$\ell_1(g)\geq \cdots \geq \ell_d(g)$) the singular values
(resp. moduli of eigenvalues) of $g$. For $1 \le i,j \le d$, we write
$\frac{\sigma_i}{\sigma_j}(g)$ for $\frac{\sigma_i(g)}{\sigma_j(g)}$
and $\frac{\ell_i}{\ell_j}(g)$ for $\frac{\ell_i(g)}{\ell_j(g)}$. A
semigroup $\Gamma \subset \GL_d(\K)$ is called \emph{$k$-divergent}
$(1 \le k < d$) if, for every sequence $\gamma_n$ of pairwise distinct
elements in $\Gamma$, we have
$\frac{\sigma_k}{\sigma_{k+1}}(\gamma_n) \to \infty$.

For semigroups $\Gamma_1, \Gamma_2\subset \mathsf{GL}_d(\mathbb{K})$
let $\langle \Gamma_1,\Gamma_2\rangle$ denote the semigroup they
generate. An expression
$\gamma_1\cdots \gamma_n\in \langle \Gamma_1,\Gamma_2\rangle$,
$n \geq 2$, is called a {\em reduced word} if all $\gamma_i$ are
non-trivial and no consecutive $\gamma_i$ lie in the same semigroup.

We use $\P(\K^d)$ to denote the projective space over the (real or
complex) vector space $\K^d$, and for each $1 \le k < d$ we let
$\Gr_k(\K^d)$ denote the Grassmannian of $k$-planes in $\K^d$. Recall
that the space of hyperplanes $\Gr_{d-1}(\K^d)$ is canonically
identified with the dual projective space $\P((\K^d)^\ast)$.

\begin{definition}\label{pingpong-position}
  Let $\Gamma_1, \Gamma_2 \subset \GL_d(\K)$ be semigroups, and let
  $U_1, U_2 \subset \P(\K^d)$ and $V_1, V_2 \subset \Gr_{d-1}(\K^d)$
  be open subsets. We say that $\Gamma_1, \Gamma_2$ are \emph{in
    ping-pong position relative to} $U_1, U_2$ and $V_1, V_2$ if,
  whenever $\{i, j\} = \{1, 2\}$, the following conditions hold:
  \begin{enumerate}
  \item The sets $\overline{U_i}$ and $\overline{V_j}$ are transverse.
  \item For every $\gamma \in \Gamma_i \minus \{I_d\}$, we have
    $\gamma \overline{U_j} \subset U_i$, and
    $\gamma^{-1}\overline{V_j} \subset V_i$.
  \end{enumerate}
\end{definition}

Our first main theorem is below:

\begin{theorem}
  \label{thm:pingpong_word_bound}
  Suppose that $\Gamma_1, \Gamma_2 \subset \GL_d(\K)$ are
  $1$-divergent semigroups, in ping-pong position relative to subsets
  $U_1, U_2\subset \mathbb{P}(\mathbb{K}^d), V_1, V_2\subset \mathsf{Gr}_{d-1}(\mathbb{K}^d)$. Then: 
  \begin{enumerate}[label=(\roman*)]
  \item\label{item:wordlength_bound} There exists $C_1 > 0$ and
    $\lambda > 1$ such that, for any reduced word
    $\gamma_1 \cdots \gamma_n \in \langle \Gamma_1, \Gamma_2 \rangle$,
    we have
    \[
      \frac{\sigma_1}{\sigma_2}(\gamma_1 \cdots \gamma_n) \geq
      C_1\lambda^n.
    \]
  \item\label{item:pingpong_1_bound} There exists $C_2 > 0$ such that,
    for any reduced word
    $\gamma_1 \cdots \gamma_n \in \langle \Gamma_1, \Gamma_2 \rangle$,
    we have
    \[
      \sigma_1(\gamma_1 \cdots \gamma_n) \geq C_2^n \sigma_1(\gamma_1)
      \cdots \sigma_1(\gamma_n).
    \]
  \item\label{item:pingpong_12_bound} There exists $C_3 > 0$ such
    that, for any reduced word
    $\gamma_1 \cdots \gamma_n \in \langle \Gamma_1, \Gamma_2 \rangle$,
    we have
    \[
      \frac{\sigma_1}{\sigma_2}(\gamma_1 \cdots \gamma_n) \geq  C_3^n
      \frac{\sigma_1}{\sigma_2}(\gamma_1) \cdots
      \frac{\sigma_1}{\sigma_2}(\gamma_n).
    \]
    \item\label{item:pingpong_12_bound-eigen} There exists $C_4 > 0$ such
    that, for every $n\in \mathbb{N}$ even and any reduced word
    $\gamma_1 \cdots \gamma_n \in \langle \Gamma_1, \Gamma_2 \rangle$,
    we have
    \[
      \frac{\ell_1}{\ell_2}(\gamma_1 \cdots \gamma_n) \geq C_4^n
      \frac{\sigma_1}{\sigma_2}(\gamma_1) \cdots
      \frac{\sigma_1}{\sigma_2}(\gamma_n).
    \]
  \end{enumerate}
\end{theorem}

In a typical application of \Cref{thm:pingpong_word_bound}, the
semigroup $\langle \Gamma_1, \Gamma_2 \rangle$ is naturally isomorphic
to the free product of semigroups $\Gamma_1 \ast \Gamma_2$, and in
fact this always holds whenever $\Gamma_1$ and $\Gamma_2$ are both
groups. See \Cref{prop:free_product_faithful}.

Recall that a representation $\rho:H \to \GL_d(\K)$ is a
\emph{quasi-isometric embedding} if there exist $C \ge 0$ and
$c \ge 1$ such that for every $h \in H$,
\[
  c^{-1}|h|_{\mathsf{H}}-C\leq \log
  \frac{\sigma_1}{\sigma_d}(\rho(h))\leq c|h|_{\mathsf{H}}+C,
\]
where $|h|_H$ is the length of $h$ with respect to a choice of finite
generating set for $H$. One consequence of the estimates in Theorem
\ref{thm:pingpong_word_bound} is the following.

\begin{corollary}
  \label{cor:pingpong_qi_embed}
  Let $\Gamma_1, \Gamma_2 < \GL_d(\K)$ be two $1$-divergent subgroups
  which are in ping-pong position relative to subsets
  $U_1, U_2, V_1, V_2$, so that $\langle \Gamma_1, \Gamma_2 \rangle$
  is naturally identified with the free product
  $\Gamma_1 \ast \Gamma_2$. If both $\Gamma_1, \Gamma_2 < \GL_d(\K)$
  are quasi-isometrically embedded, then so is
  $\langle \Gamma_1, \Gamma_2 \rangle$.
\end{corollary}

The notion of an Anosov semigroup was introduced in \cite[\S
5]{KP}. As another corollary of Theorem \ref{thm:pingpong_word_bound}
we obtain the following result.

\begin{corollary}\label{semigroup-Anosov} Let
  $\Gamma<\mathsf{GL}_d(\mathbb{K})$ be an $1$-Anosov
  subgroup. Suppose that the proximal limit set of $\Gamma$ and its
  dual $\Gamma^{\ast}$ are contained (in possibly different) affine
  charts of $\mathbb{P}(\mathbb{K}^d)$. There is an element
  $g\in \mathsf{GL}_d(\mathbb{K})$ and a finite-index subgroup
  $\Gamma'<\Gamma$ such that the semigroup $\langle \Gamma', g\rangle$
  is $1$-Anosov and isomorphic to the free product of the semigroups
  $\Gamma'$ and $\{g^n:n\geq 0\}$.
\end{corollary}

This corollary allows us to produce non-elementary examples of
$1$-Anosov semigroups in $\GL_d(\Real)$ which do \emph{not} generate
$1$-Anosov subgroups; see \Cref{sec:anosov_semigroup_not_subgroup}.

\begin{remark}
  Using techniques similar to the proof of \Cref{semigroup-Anosov},
  one can also use \Cref{thm:pingpong_word_bound} to recover special
  cases of combination theorems for Anosov subgroups due to
  Dey-Kapovich \cite{Dey-Kapovich,DK23} and
  Danciger-Gu\'eritaud-Kassel \cite{DGK24}. Since Anosov subgroups
  have uniform eigenvalue and singular value gap growth properties,
  these combination theorems themselves recover estimates as in
  \Cref{thm:pingpong_word_bound}---but \emph{only} in the case where
  the combined group is Anosov, which need not be the case for general
  applications of \Cref{thm:pingpong_word_bound}.

  We remark that the techniques in \cite{DK23} and \cite{DGK24} also
  establish combination theorems for nontrivial amalgams and HNN
  extensions of Anosov subgroups. We do not provide estimates as in
  \Cref{thm:pingpong_word_bound} for such combinations in this paper,
  although we expect that our methods also apply in this case.
\end{remark}

\subsection{Singular value gap estimates for arbitrary free products}

Our second main theorem applies more broadly, to arbitrary subgroups
of $\SL_d(\Real)$. We do not explicitly assume any ping-pong type
configuration of open subsets in the statement of this theorem; such a
configuration instead arises in the course of the proof.

\begin{theorem}\label{discrete-freeprod-1} For any
  $d\in \mathbb{N}_{\geq 2}$, there exists a representation
  \hbox{$\rho:\mathsf{SL}_d(\mathbb{R})\ast
    \mathsf{SL}_d(\mathbb{R})\rightarrow \mathsf{GL}_m(\mathbb{R})$},
  $m=\frac{d^2(d+1)^2}{4}+1$, with the following property. If
  $\Gamma_1,\Gamma_2<\mathsf{SL}_d(\mathbb{R})$ are discrete finitely
  generated subgroups, then there exist finite-index subgroups
  $H_1<\Gamma_1$ and $H_2<\Gamma_2$ such that the restriction of
  $\rho$ to $H_1\ast H_2$ is discrete and faithful. In addition, for
  any $\eps \in (0, 1)$, the finite-index subgroups $H_1, H_2$ can be
  chosen so that for every reduced word
  $\gamma_1\cdots \gamma_n\in H_1\ast H_2$, the following estimate
  holds:
\[
\frac{\sigma_1}{\sigma_m} \left(\rho(\gamma_1\cdots \gamma_n)\right)\geq  \prod_{i=1}^n \left(\frac{\sigma_1}{\sigma_d}(\gamma_i)\right)^{2-\varepsilon}.
\]
\end{theorem}

\begin{remark}
  The first part of \Cref{discrete-freeprod-1} (concerning the
  existence of finite-index subgroups so that the representation
  $\rho:H_1 \ast H_2 \to \GL_m(\Real)$ is faithful and discrete)
  follows from work of Danciger-Gu\'eritaud-Kassel \cite{DGK24}, and
  when we prove \Cref{discrete-freeprod-1} we will use their
  techniques to show that this part holds. Thus our main contribution
  in \Cref{discrete-freeprod-1} is the second part, giving the
  singular value gap estimate.
\end{remark}

As a consequence of the estimate in \Cref{discrete-freeprod-1}, we
establish that the class of linear groups admitting representations
which are quasi-isometric embeddings are closed under free products.

\begin{corollary}
  \label{cor:qi_free_products}
  Suppose that $\Gamma_1,\Gamma_2<\mathsf{GL}_d(\mathbb{R})$ are
  finitely generated and quasi-isometrically embedded subgroups. Then
  there exists $r\in \mathbb{N}$ and a faithful representation
  $\rho:\Gamma_1\ast \Gamma_2\rightarrow \mathsf{GL}_r(\mathbb{R})$
  which is a quasi-isometric embedding.
\end{corollary}

\subsection{Other semisimple Lie groups} In this paper we only
explicitly consider discrete subgroups of $\GL_d(\K)$ for $\K = \Real$
or $\C$. However, it is still possible to indirectly apply our results
towards discrete subgroups of an arbitrary semisimple Lie group $G$,
by mapping these discrete subgroups into $\GL_d(\K)$ via some
representation $G \to \GL_d(\K)$. We give a rough idea of this below.

If $P$ is a parabolic subgroup of $G$, then one can consider
\emph{$P$-divergent} groups $\Gamma \subset G$; for an appropriate
choice of representation $G \to \GL_d(\K)$, a group $\Gamma$ is
$P$-divergent if and only if its image in $\GL_d(\K)$ is $1$-divergent
(see e.g. \cite[Sec. 3]{GGKW}). One can consider $P$-divergent
subgroups $\Gamma_1, \Gamma_2$ of $G$ which are in ping-pong position
with respect to open subsets of the \emph{flag manifold} $G/P$, as
well as its natural opposite. One can show that in this situation, it
is possible to pass to finite-index subgroups $\Gamma_1', \Gamma_2'$
of $\Gamma_1, \Gamma_2$, whose images in $\GL_d(\K)$ are in ping-pong
position with respect to open subsets of $\P(\K^d)$ and
$\Gr_{d-1}(\K^d)$. Thus, one may apply \Cref{thm:pingpong_word_bound}
to obtain estimates for the behavior of the Jordan and Cartan
projections of elements in the free product
$\Gamma_1' \ast \Gamma_2'$, with respect to the roots of $G$
determined by the parabolic subgroup $P$.

\subsection*{Acknowledgements} The authors would like to thank
Subhadip Dey for several helpful comments and corrections. The second
author was partially supported by NSF grant DMS-2202770.
 
\section{Preliminaries}
\label{sec:preliminaries}

\subsection{Geometry of Grassmannians over $\K^d$}
Let $\mathbb{K}=\mathbb{R}$ or $\mathbb{C}$. For $d \ge 2$,
$(e_1,\ldots,e_d)$ is the canonical basis of $\mathbb{K}^d$ equipped
with the standard (Hermitian) inner product
$\langle \cdot,\cdot\rangle$, and $\Sph(\mathbb{K}^d)$ is the unit
sphere in $\mathbb{K}^d$ with respect to this inner product. We denote
by $\mathsf{K}_d$ the maximal compact subgroup of
$\mathsf{GL}_d(\mathbb{K})$ preserving the inner product
$\langle \cdot,\cdot \rangle$ (i.e. $\mathsf{K}_d=\mathsf{O}(d)$ if
$\mathbb{K}=\mathbb{R}$ and $\mathsf{K}_d=\mathsf{U}(d)$ if
$\mathbb{K}=\mathbb{C}$).

 Given a unit vector $v\in \mathbb{K}^d$, $v^{\perp}=\{\omega\in
 \mathbb{K}^d:\langle \omega ,v\rangle=0\}$ denotes the orthogonal
 complement of the vector space spanned by $v$. We equip the
 projective space $\mathbb{P}(\mathbb{K}^d)$ with the metric
 \[
   \dpr([u],[v]) = \sqrt{1-\frac{|\langle
       u,v\rangle|^2}{||u||^2||v||^2}}.
 \]
 For any $u, v \in \K^d$, a straightforward computation gives us the
 following formula for $\dpr([u], [v])$ in terms of the canonical
 basis $(e_1, \ldots, e_d)$:
 \begin{align}
   \label{eq:metric_formula}
   \dpr([u], [v])^2 = \frac{\sum_{1 \le i < j \le k} |\langle u, e_i
     \rangle \langle v, e_j \rangle - \langle u, e_j \rangle \langle
     v, e_i \rangle|^2}{\left(\sum_{i=1}^d |\langle u, e_i \rangle|^2
     \right) \left(\sum_{i=1}^d |\langle v, e_i \rangle|^2\right)}.
 \end{align}
 We equip the Grassmannian $\Gr_{d-1}(\K^d)$ with the metric $\dgr$
 defined by viewing elements of $\Gr_{d-1}(\K^d)$ as subsets of
 $\P(\K^d)$, and taking Hausdorff distance with respect to
 $\dpr$. Equivalently, if $u^\perp$ and $v^\perp$ are hyperplanes in
 $\Gr_{d-1}(\K^d)$ for $u, v \in \Sph(\K^d)$, then
 \[
   \dgr(u^{\perp},v^{\perp}) =d_{\mathbb{P}}([u],[v]).
 \]
 For a point $x\in \mathbb{P}(\mathbb{K}^d)$ (resp.
 $V\in \mathsf{Gr}_{d-1}(\mathbb{K}^d)$) and $\varepsilon>0$, let
 $B_{\eps}(x)$ and $B_\eps(V)$ respectively denote the open balls of
 radius $\eps$ about $x$ and $V$, with respect to the metrics $\dpr$
 and $\dgr$.

 If $x$ is an element of $\Gr_k(\K^d)$, and $y$ is an element of $\Gr_j(\K^d)$, with $j + k \le d$, then we let $\dist{x,y}$ denote the minimum distance between $x$ and $y$ when both are viewed as subsets of projective space $\P(\K^d)$. In the special case where $X$ is a singleton $\{[u]\}$ for
 $u \in \Sph(\K^d)$, and $Y$ is a $k$-dimensional projective subspace
 $\P(V)$, then $\dist{[u], \P(V)}$ is the length of the projection of
 $u$ onto $V^\perp$. In particular, if $u = k_1e_1$ and
 $V = k_2e_1^\perp$ for $k_1, k_2 \in \ms{K}_d$, then
 \begin{align}
   \label{eq:dist_to_hyperplane}
   \dist{[u], \P(V)} = |\langle k_1e_1, k_2e_1 \rangle|.
 \end{align}
 For $[u],[v]\in \mathbb{P}(\mathbb{K}^d)$ and
 $V,W\in \mathsf{Gr}_{d-1}(\mathbb{K}^d)$, we shall use repeatedly the
 following estimates from the triangle inequality:
 \begin{align*}
   \big| \dist{[u], \mathbb{P}(V)} - \dist{[v], \mathbb{P}(V)}\big|
   &\leq d_{\mathbb{P}}\left([u],[v]\right)\\
   \big| \dist{[u], \mathbb{P}(V)}-\dist{[v], \mathbb{P}(W)}\big|
   &\leq d_{\mathsf{Gr}}\left(V,W\right).
 \end{align*}

 \subsection{Singular values and the Cartan projection}

 For a matrix $g \in \mathsf{GL}_d(\mathbb{K})$ let
 $\sigma_{1}(g) \geqslant \sigma_{2}(g) \geqslant \ldots \geqslant
 \sigma_{d}(g)$ (resp.
 $\ell_{1}(g) \geqslant \ell_{2}(g) \geqslant \ldots \geqslant
 \ell_{d}(g)$) be the singular values (resp. moduli of eigenvalues) of
 $g$ in non-increasing order. Recall that for each
 $i\in \{1,\ldots, d-1\}$, $\sigma_{i}(g)=\sqrt{\ell_{i}(gg^{\ast})}$,
 where $g^{\ast}=\overline{g}^t$ is the conjugate transpose of $g$.

 Recall that $\GL_d(\K)$ has a \emph{Cartan} or \emph{KAK
   decomposition}
\[
\mathsf{GL}_d(\mathbb{K})=\mathsf{K}_d\exp(\mathfrak{a}^{+})\mathsf{K}_d,
\]
where
$\mathfrak{a}^{+}=\big\{\textup{diag}(a_1,\ldots,a_d):a_1\geq a_2\geq
\cdots \geq a_d\big\}$. This means that for any $g \in \GL_d(\K)$,
there exist $k_g, k_g' \in \ms{K}_d$ so that
\[
  g = k_g\exp(\mu(g))k_g',
\]
where $\mu:\GL_d(\K) \to \Real^d$ is the \emph{Cartan projection}
defined by
\[
  \mu(g) := (\log \sigma_1(g), \ldots, \log \sigma_d(g)).
\]

Given $i\in \{1,\ldots, d-1\}$, we say that a matrix
$g\in \mathsf{GL}_d(\mathbb{K})$ has a {\em gap of index $i$}, if
$\sigma_i(g)>\sigma_{i+1}(g)$. In this case, if we write
$g=k_g\textup{diag}(\sigma_1(g),\ldots,\sigma_d(g))k_g'$ for
$k_g,k_g'\in \mathsf{K}_d$, then the subspace
\[
\Xi_i(g):=k_g\langle e_1,\ldots, e_i\rangle
\]
is well-defined and does not depend on the choice of
$k_g, k_g' \in \ms{K}_d$.

A semigroup $\Gamma\subset \mathsf{GL}_d(\mathbb{K})$ is {\em
  i-divergent} if for every infinite sequence of distinct elements
$(\gamma_n)_{n\in \mathbb{N}}\subset \Gamma$ we have
\[
\lim_{n\rightarrow \infty}\frac{\sigma_i}{\sigma_{i+1}}(\gamma_n)=+\infty.
\]
Note that if the semigroup $\Gamma$ is $i$-divergent, the inverse
semigroup $\Gamma^{-1}$ is $(d-i)$-divergent. For an $i$-divergent
semigroup $\Gamma\subset \mathsf{GL}_d(\mathbb{K})$ denote by
$\Lambda_i(\Gamma)$ its {\em limit set} in
$\mathsf{Gr}_i(\mathbb{K}^d)$, i.e. the set of all accumulation points
of the sequence $(\Xi_i(\gamma_n))_{n\in \mathbb{N}}$ as
$(\gamma_n)_{n\in \mathbb{N}}\subset \Gamma$ runs over all infinite
sequences of distinct elements in $\Gamma$. For more details on
definitions we refer to \cite{GGKW, KLP2}.

\subsection{Initial estimates}

We now prove some basic lemmas that we will use throughout the proofs
of our main results. Some (but not all) of these lemmas are almost
identical to results appearing in the appendix of \cite{BPS}. However,
in our case, we work in $\K^d$, rather than $\Real^d$, and we do not
need to assume that the matrices we consider have any gap between
their singular values. For the reader's convenience we provide short
proofs for all of our estimates.

The first two lemmas we need are contraction results, relating
singular value gaps of elements in $\GL_d(\K)$ to radii of nested
balls in $\P(\K^d)$. \Cref{Lipschitz} shows that, any element
$g \in \GL_d(\K)$ always maps some ball $B_1$ in projective space
$\P(\K^d)$ into some other ball $B_2$, with the ratio of the radii of
$B_1$ and $B_2$ depending on the singular value gap
$\frac{\sigma_1}{\sigma_2}(g)$ and the distance from $B_1$ to the
space $\Xi_{d-1}(g^{-1})$. \Cref{gap-estimate1} gives a converse
result: if we know that some element $g \in \GL_d(\K)$ maps some ball
$B_1$ in $\P(\K^d)$ into some other ball $B_2$, then we have an
estimate on $\frac{\sigma_1}{\sigma_2}(g)$ in terms of the radii of
$B_1$ and $B_2$, and the distance from $B_1$ to $\Xi_{d-1}(g^{-1})$.
\begin{lemma}\label{Lipschitz}
  Let $g\in \mathsf{GL}_d(\mathbb{K})$ and write
  $g:=k_g\exp(\mu(g))k_g'$ for $k_g,k_g'\in \mathsf{K}_d$. Fix
  $0<\varepsilon<1$ and $x\in \mathbb{P}(\mathbb{K}^d)$ with
  $\dist{x,\mathbb{P}((k_g')^{-1}e_1^{\perp})}\geq
  \varepsilon.$ Then for every $0<\delta \leq \frac{1}{2} \varepsilon$
  we have that 
\[
 gB_{\delta}(x) \subset B_{\delta'}(gx), \ 
\]
 where
  $\delta':=\frac{2\delta}{\varepsilon}\frac{\sigma_2}{\sigma_1}(g)$.
\end{lemma}

\begin{proof} First observe that since
  $\dist{x, \mathbb{P}((k_g')^{-1}e_1^{\perp})} \geq \varepsilon$, for
  every $y\in B_{\delta}(x)$ we have that
  \[
    \dist{y, \mathbb{P}((k_g')^{-1}e_1^{\perp})} \geq
    \varepsilon-d_{\mathbb{P}}(x,y)\geq \frac{\varepsilon}{2}.
  \]
  Write $x=[u_x], y=[u_y]$ for
  $u_x,u_y\in \mathsf{S}^1(\mathbb{K}^d)$. We can use
  \eqref{eq:metric_formula} to see that
  \begin{align*}
    d_{\mathbb{P}}\left(gx,gy\right)^2
    &= \frac{\sum_{1\leq i<j\leq d}\sigma_i(g)^2\sigma_j(g)^2
      \big|\langle k_g'u_x,e_i\rangle \langle
      k_g'u_y,e_j\rangle-\langle k_g' u_x,e_j\rangle \langle k_g'
      u_y,e_i\rangle \big|^2}{\left(\sum_{i=1}^{d}\sigma_i(g)^2
      \big|\langle k_g' u_x,e_i\rangle\big|^2\right)
      \left(\sum_{i=1}^{d}\sigma_i(g)^2 \big|\langle k_g'
      u_y,e_i\rangle\big|^2\right)}\\&\leq
    \frac{\sigma_1(g)^2\sigma_2(g)^2}{\sigma_1(g)^4}\frac{d_{\mathbb{P}}(x,y)^2}{\big|\langle
    k_g' u_x,e_1\rangle \big|^2 \big|\langle k_g' u_y,e_1\rangle
    \big|^2}.
  \end{align*}
  Using \eqref{eq:dist_to_hyperplane}, we know that
  $\varepsilon/2 \leq \dist{y, \P((k_g')^{-1}e_1^\perp)} =
  \left|\langle k_g'u_y, e_1 \rangle \right|^2$,
  so we can bound the above beneath
  \begin{align*}
    \frac{4\sigma_2(g)^2}{\sigma_1(g)^2}\frac{1}{\varepsilon^{2}}d_{\mathbb{P}}(x,y)^2\leq
    \frac{4\sigma_2(g)^2}{\sigma_1(g)^2}\frac{\delta^2}{\varepsilon^{2}}.
  \end{align*}
  The lemma follows.
\end{proof}

\begin{lemma}\label{gap-estimate1} Let $g\in
  \mathsf{GL}_d(\mathbb{K})$ written as $g=k_g\exp(\mu(g))k_g'$ for
  $k_g,k_g'\in \mathsf{K}_d$, and let $x \in \P(\K^d)$. Suppose that
  $0<\delta, r<1$ satisfy $gB_{\delta}(x)\subset B_{r}(gx)$. Then the
  following estimate holds:
  \[
    \frac{\sigma_1}{\sigma_2}(g)\geq \frac{\delta}{4r}
    \dist{x,\mathbb{P}((k_g')^{-1}e_1^{\perp})}
  \]
\end{lemma}

\begin{proof}
  After conjugation by $k_g'$, we may assume without loss of
  generality that $k_g' = \identity$.  Fix a unit vector
  $u_x\in \mathsf{S}^1(\mathbb{K}^d)$ with $x=[u_x]$.  If
  $\dist{x, \mathbb{P}(e_1^{\perp})}=\big|\langle
  u_x,e_1\rangle\big|=0$, there is nothing to prove. So we may assume
  that $\big|\langle u_x,e_1\rangle\big|>0$ and
  $\big|\langle u_x,e_2\rangle \big|<1$.

  For each $t \in (0, 1)$, consider the unit vector
  \[
    v(t) := \frac{u_x+t e_2}{\big|\big|u_x+t e_2\big|\big|}.
  \]
  A computation shows that
  \begin{align*}
    d_{\mathbb{P}}\left([u_x],[v(t)]\right)
    &=\sqrt{1-\big|\langle u_x, v(t)\rangle
      \big|^2}=\frac{t\sqrt{1-\big|\langle u_x, e_2\rangle\big|^2}}{\big|\big|u_x+te_2\big|\big|}\\
    &\leq \frac{t}{1-t}\sqrt{1-\big|\langle u_x,e_2\rangle
      \big|^2}.
  \end{align*}
  In particular, we have $[v(t)] \in B_{\delta}(x)$ if $t$ satisfies
  \[
    t = \frac{\delta}{\delta+\sqrt{1-\big|\langle u_x,e_2\rangle|^2}}.
  \]
  So, now set $u_y = v(t)$ for some $t$ as above. Suppose that
  $gB_\delta(x) \subset B_r(x)$, so that
  $d_{\mathbb{P}}\left([gu_x],[gu_y]\right) < r$. Using the formula
  \eqref{eq:metric_formula}, we have the estimates
  \begin{align*}
    d_{\mathbb{P}}\left([gu_x],[gu_y]\right)^2
    &=\frac{\sum_{1\leq i<j\leq d}\sigma_i(g)^2\sigma_j(g)^2
      \big|\langle u_x,e_i\rangle \langle u_y,e_j\rangle-\langle
      u_x,e_j\rangle \langle u_y,e_i\rangle
      \big|^2}{\left(\sum_{i=1}^{d}\sigma_i(g)^2 \big|\langle
      u_x,e_i\rangle\big|^2\right) \left(\sum_{i=1}^{d}\sigma_i(g)^2
      \big|\langle u_y,e_i\rangle\big|^2\right)}\\
    &\geq \frac{\sigma_2(g)^2}{\sigma_1(g)^2}\big|\langle
      u_x,e_1\rangle \langle u_y,e_2\rangle-\langle
      u_x,e_2\rangle \langle u_y,e_1\rangle \big|^2\\
    &=  \frac{\sigma_2(g)^2}{\sigma_1(g)^2}\frac{t^2 \big|\langle
      u_x,e_1\rangle\big|^2}{\big|\big|u_x+t e_2\big|\big|^2}\\
    &\geq
      \frac{\sigma_2(g)^2}{\sigma_1(g)^2}\frac{\delta^2}{16}\dist{x,\mathbb{P}(e_1^{\perp})}^2.
  \end{align*} 
  We deduce that
  \[
    \frac{\sigma_1}{\sigma_2}(g)\geq
    \frac{\delta}{4r}\dist{x, \mathbb{P}(e_1^{\perp})},
  \]
  which finishes the proof of the lemma.
\end{proof}

The next lemma also captures the idea that a group element
$g \in \GL_d(\K)$ attracts a point $x \in \P(\K^d)$ towards the
subspace $\Xi_1(g)$, by an amount related to the singular value gap
$\frac{\sigma_1}{\sigma_2}(g)$ and the distance from $x$ to
$\Xi_{d-1}(g^{-1})$.

\begin{lemma}[{Compare \cite[Lem. A.6]{BPS}}]
  \label{lem:BPS-estimate-1}
  Let $g \in \GL_d(\K)$, and write $g_1 = k_g\exp(\mu(g))k_g'$. For any $x \in \P(\K^d) \minus \P((k_g')^{-1}e_1^\perp)$ we have the estimate
  \[
    d_{\mathbb{P}}\left(gx, [k_ge_1]\right) \le \frac{\sigma_2}{\sigma_1}(g) \cdot
    \frac{1}{\dist{x, \P((k_g')^{-1}e_1^\perp)}}.
  \]
\end{lemma}
\begin{proof}
  Let $v\in \K^d$ be a unit vector with $x=[v]$. By the
  definition of the metric $\dpr$ we have that
  \begin{align*}
    \dpr(gx, [k_ge_1]) &= 1-\frac{\left|\langle \exp(\mu(g))k_g'v, e_1\rangle\right|^2}
      {||\exp(\mu(g))k_g'v||} =
      \frac{\sum_{i=2}^d\sigma_i(g)^2|k_g'v,e_i|^2}{\sum_{i=1}^d\sigma_i(g)^2|\langle
      k_gv,e_i\rangle|^2}\\
    &\leq \frac{\sigma_2(g)^2}{\sigma_1(g)^2}
      \frac{1}{|\langle k_g'v,e_1\rangle|^2} =
      \frac{\sigma_2(g)^2}{\sigma_1(g)^2}\dist{x,\P((k_g')^{-1}e_1^{\perp})}^{-2}.
   \end{align*}
 \end{proof}

 The next lemma allows us to relate the singular value gaps of a
 \emph{product} $g_1g_2$ of elements in $\GL_d(\K)$ to the singular
 values of $g_1$ and $g_2$.

 \begin{lemma}
    \label{gap-product-1} Let $w_1,w_2\in
    \mathsf{GL}_d(\mathbb{K})$. Then for $i\in \{1,2\}$ we have the
    estimate
    \[
      \frac{\sigma_1}{\sigma_2}(w_1w_2)\geq
      \frac{\sigma_d(w_i)^2}{\sigma_1(w_i)^2}\frac{\sigma_1}{\sigma_2}(w_1)\frac{\sigma_1}{\sigma_2}(w_2).
    \]
\end{lemma}
\begin{proof} Since $\sigma_1(\wedge^2g) = \sigma_1(g)\sigma_2(g)$ for
  all $g \in \GL_d(\K)$, we know that
  \[
    \frac{\sigma_1(w_1w_2)}{\sigma_2(w_1w_2)}=\frac{\sigma_1^2(w_1w_2)}{\sigma_1(\wedge^2
      w_1w_2)}.
  \]
  For $i = 1,2$, we know that
  $\sigma_1^2(w_1w_2) \ge \sigma_1(w_{3-i})^2\sigma_d(w_i)^2$, meaning
  we have
  \[
    \frac{\sigma_1^2(w_1w_2)}{\sigma_1(\wedge^2 w_1w_2)} \ge
    \frac{\sigma_1(w_i)\sigma_1(w_{3-i})^2\sigma_d(w_i)^2}{\sigma_1(w_i)\sigma_1(\wedge^2w_1)\sigma_1(\wedge^2w_2)}=\frac{\sigma_d(w_i)^2}{\sigma_1(w_i)^2}\frac{\sigma_1(w_1)\sigma_1(w_2)}{\sigma_2(w_1)\sigma_2(w_2)}.
  \]
\end{proof}

The next two lemmas also give us an estimate on the singular value
gaps of a product $g_1g_2$, but this time our estimates also involve
the relative positions of certain subspaces $\Xi_k(g_1)$,
$\Xi_k(g_2)$, and $\Xi_k(g_1g_2)$.
\begin{lemma}[{Compare \cite[Lem. A.7]{BPS}}]
  \label{ratio-sigma1}
  Let $g_1,g_2\in \GL_d(\mathbb{K})$ and for $i \in \{1,2\}$ write
  $g_i=k_{g_i}\exp(\mu(g_i))k_{g_i}'$, with
  $k_{g_i},k_{g_i}'\in \mathsf{K}_d$. Then the following estimate
  holds:
  \[
    \frac{\sigma_1(g_1g_2)}{\sigma_1(g_1)\sigma_1(g_2)}\geq
    \dist{[k_{g_2}e_1],\mathbb{P}((k_{g_1})^{-1}e_{1}^{\perp}))}.
  \]
\end{lemma}
\begin{proof} A computation shows that \begin{align*}
      \frac{\sigma_1(g_1g_2)}{\sigma_1(g_1)\sigma_1(g_2)}&=\Bigg|\Bigg|
      k_{g_1}\textup{diag}\left(1,
        \frac{\sigma_2}{\sigma_1}(g_1),\ldots,
        \frac{\sigma_d}{\sigma_1}(g_1)\right)k_{g_1}'k_{g_2}\textup{diag}\left(1,
        \frac{\sigma_2}{\sigma_1}(g_2),\ldots,
        \frac{\sigma_d}{\sigma_1}(g_2)\right)k_{g_2}'\Bigg|\Bigg|\\
      &=\Bigg|\Bigg|\textup{diag}\left(1,
        \frac{\sigma_2}{\sigma_1}(g_1),\ldots,
        \frac{\sigma_d}{\sigma_1}(g_1)\right)k_{g_1}'k_{g_2}\textup{diag}\left(1,
        \frac{\sigma_2}{\sigma_1}(g_2),\ldots,
        \frac{\sigma_d}{\sigma_1}(g_2)\right)\Bigg|\Bigg|\\ &\geq
      \Bigg|\Bigg| \textup{diag}\left(1,
        \frac{\sigma_2}{\sigma_1}(g_1),\ldots,
        \frac{\sigma_d}{\sigma_1}(g_1)\right)k_{g_1}'k_{g_2}e_1\Bigg|\Bigg|\\
      &\geq \big|\langle k_{g_2}'k_{g_1}e_1,e_1\rangle
      \big|=\dist{[k_{g_2}e_1],
        \mathbb{P}((k_{g_1})^{-1}e_1^{\perp})}.\end{align*}
  \end{proof}

\begin{lemma}[{Compare \cite[Lem. A.9]{BPS}}]
  \label{BPS-estimate-2} Let $g_1,g_2\in
  \mathsf{GL}_d(\mathbb{K})$ and
  write
  \begin{align*}
    g_1g_2&=k_{g_1g_2}\exp\left(\mu(g_1g_2)\right)k_{g_1g_2}', \qquad
            \textrm{ for }k_{g_1g_2},k_{g_1g_2}'\in \mathsf{K}_d,\\
    g_i&=k_{g_i}\exp\left(\mu(g_i)\right)k_{g_i}',\quad
         \textrm{ for }k_{g_i},k_{g_i}'\in \mathsf{K}_d,\ \textrm{ and
         }i=1,2.
  \end{align*}
  Then the following estimates hold:
  \begin{enumerate}[label=(\roman*)]
  \item\label{item:xi1_attract_inequality}
    \[
      \dpr\left([k_{g_1g_2}e_1],[k_{g_1}e_1]\right) \leq
      \sqrt{d-1}\frac{\sigma_2(g_1)\sigma_1(g_2)}{\sigma_1(g_1g_2)}\leq
      \frac{\sigma_2}{\sigma_1}(g) \cdot
      \frac{\sqrt{d-1}}{\dist{[k_{g_2}e_1],\mathbb{P}\left((k_{g_1}')^{-1}e_1^{\perp}\right)}},
    \]
  \item\label{item:xid1_attract_inequality}
    \[
      \dgr\left(k_{g_1g_2}e_{d}^{\perp},k_{g_1}e_d^{\perp}\right) \leq
      \sqrt{d-1}\frac{\sigma_d(g_1g_2)}{\sigma_d(g_2)\sigma_{d-1}(g_1)}\leq
      \frac{\sigma_d}{\sigma_{d-1}}(g_1) \cdot
      \frac{\sqrt{d-1}}{\dist{[(k_{g_1}')^{-1}e_d],\mathbb{P}(k_{g_2}e_{d}^{\perp})}}.
    \]
  \end{enumerate}
\end{lemma}

\begin{proof} To prove \ref{item:xi1_attract_inequality}, note that
\[
k_{g_1}^{-1}k_{g_1g_2}\exp(\mu(g_1g_2))=\exp(\mu(g_1))k_{g_1}'k_{g_2}\exp(\mu(g_2))k_{g_2}'(k_{g_1g_2}')^{-1}
\]
 and hence for $2\leq i \leq d$ we have \begin{align*} \sigma_1(g_1g_2)\big|\big\langle k_{g_1}^{-1}k_{g_1g_2}e_1,e_i\big \rangle \big|&=\big|\big\langle k_{g_1}^{-1}k_{g_1g_2}\exp(\mu(g_1g_2))e_1,e_i\big \rangle \big|\\ &=\big|\big\langle \exp(\mu(g_1))k_{g_1}'k_{g_2}\exp(\mu(g_2))k_{g_2}'(k_{g_1g_2}')^{-1}e_1,e_i\big \rangle\big|\\ & =\sigma_i(g_1)\big|\big\langle k_{g_1}'k_{g_2}\exp(\mu(g_2))k_{g_2}'(k_{g_1g_2}')^{-1}e_1,e_i\big \rangle\big|\\ & \leq \sigma_i(g_1)\sigma_1(g_2).  \end{align*} 

  Then as
  \[
    1 = ||k_{g_1}^{-1}k_{g_1g_2}e_1||^2 = \sum_{i=1}^d \left| \langle k_{g_1}^{-1}k_{g_1g_2} e_1, e_i
      \rangle \right|^2
    = \left|\langle k_{g_1}^{-1}k_{g_1g_2}e_1, e_1 \rangle\right|^2 + \sum_{i=2}^d
    \left|\langle k_{g_1}^{-1}k_{g_1g_2} e_1, e_i \rangle\right|^2,
  \] we have that
  $d_{\mathbb{P}}\left([k_{g_1g_2}e_1],[k_{g_1}e_1]\right)\leq
  \sqrt{d-1}\frac{\sigma_2(g_1)\sigma_1(g_2)}{\sigma_1(g_1g_2)}$.
  This proves the left-hand inequality in
  \ref{item:xi1_attract_inequality}; the right-hand inequality follows
  from \Cref{ratio-sigma1}. The proof of
  \ref{item:xid1_attract_inequality} is analogous.
\end{proof}

Finally we close this subsection with the following elementary
observation:

\begin{lemma}\label{dist-hyp} Let $x\in \mathbb{P}(\mathbb{K}^d)$, $0<\theta<1$ and $V\in \mathsf{Gr}_{d-1}(\mathbb{K}^d)$. There exists $y\in B_{\theta}(x)$ such that $\textup{dist}(y,\mathbb{P}(V))\geq \frac{\theta}{2}$.\end{lemma}

\begin{proof} Without loss of generality assume $V=e_1^{\perp}$. Let
  us write $x=[v_0]$, $v_0=\sum_{i=1}^{d}x_ie_i$. If
  $|x_1|\geq \frac{\theta}{2}$ then
  $\dist{[v_0],\mathbb{P}(e_1^{\perp})}\geq \frac{\theta}{2}$ and the
  statement holds with $y = x$. So, suppose that
  $|x_1|<\frac{\theta}{2}$. Let $w_0$ be the vector
  $\sum_{i=2}^dx_ie_i$, so that $v_0 = x_1e_1 + w_0$, and consider the
  unit vector
  \[
    v_0' = \frac{\theta}{2}e_1 + \frac{\sqrt{1 - \theta^2/4}}{||w_0||}w_0.
  \]
  Note that
  $\textup{dist}([v_0'],\mathbb{P}(e_1^\perp))=\frac{\theta}{2}$ and
  also
  \begin{align*}
    \left|\langle v_0',v_0\rangle \right|
    &\ge \left|\big\langle w_0, \frac{\sqrt{1 - \theta^2/4}}{||w_0||}w_0\big\rangle\right| -
      |x_1|\frac{\theta}{2}\\
    &\ge ||w_0||\sqrt{1 - \frac{\theta^2}{4}} - \frac{\theta^2}{4}\\
    &= \sqrt{1 - |x_1|^2}\sqrt{1 - \frac{\theta^2}{4}} -
      \frac{\theta^2}{4}\\
    &\ge 1 - \frac{\theta^2}{2} \ge \sqrt{1 - \theta^2},
  \end{align*}
  hence $\dpr{[v_0'],[v_0]}\leq \theta$.
\end{proof}

\subsection{Caratheodory metrics on subsets of projective space}
\label{sec:caratheodory_metric}

We need to introduce one more technical tool before we can turn to the
proof of Theorem~\ref{thm:pingpong_word_bound}, namely a metric
defined by Zimmer \cite{Zimmer} on certain open subsets of
$\P(\K^d)$. The definition of the metric relies on a pair of
embeddings
\[
  \iota:\P(\K^d) \hookrightarrow \P(\Real^n),\ 
  \iota^*:\Gr_{d-1}(\K^d) \hookrightarrow \Gr_{n-1}(\Real^n),
\]
satisfying the property that $x \in \P(\K^d)$, $W \in \Gr_{d-1}(\K^d)$
are transverse if and only if $\iota(x)$, $\iota^*(W)$ are
transverse. When $\K = \Real$, we can take $\iota, \iota^*$ to be the
identity maps. When $\K = \C$, we identify $\C^d$ with $\Real^{2d}$,
which realizes $\P(\C^d)$ as a submanifold of $\Gr_2(\Real^{2d})$, and
then take $\iota$ to be the restriction of the Pl\"ucker embedding
$\Gr_2(\Real^{2d}) \to \P(\wedge^2\Real^{2d})$. The map $\iota^*$ is
defined dually, using the natural identifications
$\Gr_{d-1}(\C^d) \simeq \P((\C^d)^*)$ and
$\Gr_{n -1}(\Real^n) \simeq \P((\Real^n)^*)$.

Let $U$ be any nonempty open subset of $\P(\K^d)$. We say that
$U$ is a \emph{proper domain} if there exists some
$W \in \Gr_{d-1}(\K^d)$ which is transverse to every
$x \in \overline{U}$. When $U$ is a proper domain, let
$U^* \subset \Gr_{d-1}(\mathbb{K}^d)$ denote the (nonempty) set
\[
  \big\{W \in \Gr_{d-1}(\mathbb{K}^d) \st v \notin W, \ \forall [v] \in
  \overline{U}\big\}.
\]

Recall that the \emph{cross-ratio}
$\chi:\P(\Real^n)^2 \times \Gr_{n-1}(\Real^n)^2 \to \Real$ can be
defined by the formula
\[
  \chi([v_1], [v_2]; u_1^\perp, u_2^\perp) := \frac{\langle v_1, u_1
    \rangle \langle v_2, u_2 \rangle}{\langle v_1, u_2 \rangle \langle
    v_2, u_1 \rangle},
\]
whenever $[v_1], [v_2]$ are both transverse to $u_1^\perp,
u_2^\perp$. Now, when $U$ is a proper domain, define a function
$\hil_{U}:U \times U \to \Real_{\ge 0}$ as follows: for
$[v_1], [v_2] \in U$, take
\[
  \hil_U([v_1], [v_2]) := \sup_{u_1^\perp, u_2^\perp \in U^*}
  \log \big|\chi(\iota([v_1]), \iota([v_2]); \iota^*(u_1^\perp),
  \iota^*(u_2^\perp))\big|.
\]

When $U$ is a \emph{properly convex domain} in real projective space
(see \Cref{sec:properly_convex_domains}), then $\hil_U$ is, up to a
constant multiplicative factor, precisely the well-known \emph{Hilbert
  metric} on $U$. More generally, $\hil_U$ satisfies the following
properties, which we will use in the sequel:
\begin{theorem}
  \label{thm:caratheodory_metric_shrinks}
  For proper domains $U, U_1, U_2$ in $\P(\K^d)$:
  \begin{enumerate}[label=(\arabic*)]
  \item\label{item:c_metric_is_metric} The function $\hil_U$ defines a
    proper metric on $U$ which induces the subspace topology on $U$ as
    an open subset of $\P(\K^d)$.
  \item\label{item:c_metric_invariance} The metric $\hil_U$ is
    $\GL_d(\K)$-invariant, in the sense that for any $g \in \GL_d(\K)$
    and $x, y \in U$, we have
    \[
      \hil_U(x, y) = \hil_{gU}(gx, gy).
    \]
  \item\label{item:c_metric_weak_contract} If $U_1 \subseteq U_2$,
    then $\hil_{U_1} \ge \hil_{U_2}$.
  \item\label{item:c_metric_strong_contract} If
    $\overline{U_1} \subset U_2$, then there exists a constant
    $\lambda > 1$ depending only on $U_1, U_2$ so that for any
    $x, y \in U_1$, we have
    \[
      \hil_{U_1}(x,y) \ge \lambda \hil_{U_2}(x,y).
    \]
  \end{enumerate}
\end{theorem}
\begin{proof}
  \ref{item:c_metric_is_metric} is \cite[Thm. 5.2]{Zimmer}, and
  \ref{item:c_metric_invariance} and \ref{item:c_metric_weak_contract}
  are immediate from the definition of
  $\hil_U$. \ref{item:c_metric_strong_contract} is
  \cite[Prop. 7.11]{Weisman}.
\end{proof}

\begin{remark}\ 
  \begin{enumerate}[label=(\alph*)]
  \item In general, the metric $\hil_U$ need not be complete; in the
    case $\K = \Real$, $\hil_U$ is complete if and only if $U$ is a
    properly convex open subset of $\P(\Real^d)$.
  \item The construction of $\hil_U$ in \cite{Zimmer} also makes sense
    in the situation where $U$ is an open subset of some Grassmannian
    $\Gr_k(\K^d)$, for any $1 \le k < d$; however, in this paper we
    only need to work with $\hil_U$ in the situation where $k = 1$.
  \end{enumerate}
\end{remark}

\section{Estimates for 1-divergent semigroups}

In this section we prove \Cref{thm:pingpong_word_bound}, as well as
its corollaries \Cref{cor:pingpong_qi_embed} and
\Cref{semigroup-Anosov}. Throughout this section we assume that we
have two semigroups
$\Gamma_1,\Gamma_2\subset \mathsf{GL}_d(\mathbb{K})$ which are in
ping-pong position relative to subsets $U_1, U_2 \subset \P(\K^d)$,
$V_1, V_2 \subset \Gr_{d-1}(\K^d)$. Recall that this means that
whenever $\{i, j\} = \{1, 2\}$, the following conditions hold:
\begin{enumerate}[label=(P\arabic*)]
\item\label{item:pingping_i} The sets $\overline{U_i}$ and $\overline{V_j}$ are transverse.
\item\label{item:pingpong_ii} For every
  $\gamma \in \Gamma_i \minus \{I_d\}$, we have
  $\gamma \overline{U_j} \subset U_i$, and
  $\gamma^{-1}\overline{V_j} \subset V_i$.
\end{enumerate}

% Since $\Gamma_1, \Gamma_2\subset \mathsf{GL}_d(\mathbb{K})$ are in ping-pong
% position, the semigroup $\Delta:=\langle \Gamma_1,\Gamma_2 \rangle$
% they generate is naturally isomorphic to $\Gamma_1 \ast \Gamma_2$, and
% it is $1$-divergent as a consequence of
% \cite[Thm. 3.1]{Dey-Kapovich}. Throughout this section we will
% implicitly identify $\Delta$ with the abstract free product
% $\Gamma_1 \ast \Gamma_2$.

\subsection{Setup for the proof}

By condition \ref{item:pingping_i} above, we may fix $\varepsilon>0$
so that for $\{i, j\} = \{1,2\}$, we have
\begin{align}\label{est-1-}
  \dist{\overline{U}_i, \overline{V_j}} := \inf\left\{ \dist{x, y} : x \in
  \overline{U_i}, y \in \overline{V_j}\right\}
  \geq \varepsilon.
\end{align}
We also fix $x_i\in U_{i}$ (resp. $y_i\in V_i$) and
$0<\theta<\varepsilon^2$, depending only on $U_1,U_2,V_1,V_2$, such
that $B_{\theta}(x_i)\subset U_i$ (resp. $B_{\theta}(y_i)\subset V_i$)
for $i=1,2$.

Let $\Delta$ be the semigroup $\langle \Gamma_1, \Gamma_2 \rangle$. We
may view each element $g \in \Delta$ as a reduced alternating word
$\gamma_1 \cdots \gamma_n$. The correspondence between such
alternating words and elements of $\Delta$ need not be
bijective. However, this does not matter for any of the arguments
below, so we will frequently identify elements in $\Delta$ with such
reduced words.

Let
$M = \max\left\{\frac{16}{\varepsilon
    \theta},\frac{8\sqrt{d-1}}{\varepsilon^2}\right\}$, and let
$F \subset \Delta$ be the set of all elements $h \in \Delta$
satisfying $\frac{\sigma_1}{\sigma_2}(h) < M$; in other words $F$ is a
minimal subset such that $\Delta \minus F$ satisfies
\begin{align}\label{lower-1}
  \frac{\sigma_1}{\sigma_2}(h) \ge \max\left\{\frac{16}{\varepsilon
  \theta},\frac{8\sqrt{d-1}}{\varepsilon^2}\right\}\ \forall h \in
  \Delta \minus F.
\end{align}

Later we will show that $\Delta$ is $1$-divergent, which means that
$F$ is actually a finite set. For now, since we know that $\Gamma_1$
and $\Gamma_2$ are $1$-divergent, we know that each of
$\Gamma_1 \cap F$ and $\Gamma_2 \cap F$ is finite.

\begin{lemma}\label{estimate-Xi}
  Let $g \in \Delta \minus F$ and $g = g_1\cdots g_n$, $n\geq 1$ be a
  reduced word in $\Delta \minus F$, with
  $g_i\in \Gamma_1\cup \Gamma_2\minus \{I_d\}$.
  \begin{enumerate}[label=(\roman*)]
  \item if $i(g_1)\in \{1,2\}$ is the unique index such that
    $g_1\in \Gamma_{i(g_1)}$, we have
    \begin{align*}
      \dpr\left(\Xi_{1}(g), U_{i(g_1)}\right) \leq \frac{\varepsilon}{8}.
    \end{align*}
  \item if $i(g_n)\in \{1,2\}$ is the unique index such that
    $g_n\in \Gamma_{i(g_n)}$, we have
    \begin{align*}
      \dgr\left(\Xi_{d-1}(g^{-1}), V_{i(g_n)} \right)\leq
      \frac{\varepsilon}{8}.
    \end{align*}
  \end{enumerate}
\end{lemma}
\begin{proof} By the ping-pong conditions \ref{item:pingping_i} and
  \ref{item:pingpong_ii} we have that $gU_j\subset U_i$, where
  $i = i(g_n)\in \{1,2\}$ is the unique index such that
  $g_1\in \Gamma_{i}$ and $j = j(g_n)\in \{1,2\}$ is the unique index
  such that $g_n\in \Gamma_{3-j}$. In particular, by Lemma
  \ref{dist-hyp} we may choose $x\in B_{\theta}(x_j)$ such that
  $\textup{dist}(x,\Xi_{d-1}(g^{-1}))\geq \frac{\theta}{2}$. Thus,
  since $\frac{\sigma_1}{\sigma_2}(g)>16(\varepsilon \theta)^{-1}$, by
  \Cref{lem:BPS-estimate-1} we have that
  $$d_{\mathbb{P}}\left(gx, \Xi_{1}(g)\right)\leq
  \frac{2}{\theta}\frac{\sigma_2}{\sigma_1}(g)\leq
  \frac{\varepsilon}{8}.$$

  Since $gx\in U_{i(g_1)}$, part (i) follows. The proof of (ii) is
  analogous.
\end{proof}

\subsection{Proof of \Cref{thm:pingpong_word_bound}}

We begin with the proof of \Cref{thm:pingpong_word_bound} part
\ref{item:wordlength_bound}. First, we observe:
\begin{lemma}
  \label{lem:eps_nbhd_pingpong}
  Let $\{i,j\} = \{1,2\}$. There exists some $\delta > 0$ so that
  $N_{\delta}(U_i)$ and $N_{\delta}(V_j)$ are proper domains (see
  \Cref{sec:caratheodory_metric}), and so that for every
  $\gamma \in \Gamma_i \minus \{I_d\}$, we have
  \begin{align*}
    \gamma N_{\delta}(U_j) \subset N_{\delta/2}(U_i),\\
    \gamma^{-1}N_{\delta}(V_j) \subset N_{\delta/2}(V_i).
  \end{align*}
\end{lemma}
\begin{proof}
  The ping-pong conditions imply that both $U_i$ and $V_i$ are proper
  domains. Then since $\overline{U_i}$ and $\overline{V_i}$ are
  compact, and transversality is an open condition, there exists a
  $\delta$ so that $N_\delta(U_i)$ and $N_\delta(V_i)$ are proper.

  So, we just need to show that there is some $\delta$ so that the
  desired inclusions hold. We will only show this for the sets
  $U_i, U_j$, since the inclusions for $V_i, V_j$ are analogous. The
  first step is to show that there is some $\delta > 0$ and a finite
  subset $F' \subset \Gamma_i$ so that so that the desired inclusions
  hold for all $\gamma \in \Gamma_i \minus F'$. Let
  $\delta = \eps / 2$, where $\eps$ is the constant from
  \eqref{est-1-}, and let $F \subset \Delta$ be the subset defined in
  \eqref{lower-1}. We choose a finite set $F' \subset \Gamma_i$ large
  enough to contain $F \cap \Gamma_i$, and so that
  $\frac{\sigma_1}{\sigma_2}(g) \ge 64/3\eps^2$ for all
  $g \in \Gamma_i \minus F'$.

  Fix $x \in N_{\delta}(U_j)$ and $g \in \Gamma_i \minus F'$. By
  \Cref{estimate-Xi} we know that
  \[
    \dgr\left(\Xi_{d-1}(g^{-1}), V_i\right) \le \frac{\eps}{8},
  \]
  and since $\dpr(x, U_j) \le \eps/2$ and
  $\dist{\overline{U_j}, \overline{V_i}} \ge \eps$, we have therefore
  \[
    \dist{\Xi_{d-1}(g^{-1}), x} \ge \frac{3\eps}{8}.
  \]
  Then, as $\frac{\sigma_1}{\sigma_2}(g) \ge 64/3\eps^2$,
  \Cref{lem:BPS-estimate-1} implies that
  \[
    \dpr(gx, \Xi_1(g)) \le \frac{\eps}{8}.
  \]
  Then we can apply \Cref{estimate-Xi} again to see that
  $\dpr(gx, U_i) \le \eps/4 = \delta/2$.

  This shows that for all $\gamma \in \Gamma_i \minus F'$, we have
  $\gamma N_\delta(U_j) \subset N_{\delta/2}(U_i)$. Now, since
  $\overline{U_j}$ is compact, the ping-pong condition
  \ref{item:pingpong_ii} implies that for every
  $\gamma \in F' \minus \{\identity\}$, there is some
  $\delta(\gamma) > 0$ so that
  $\gamma N_{\delta(\gamma)}(U_j) \subset U_i$. So, by replacing
  $\delta$ with the minimum of $\delta$ and
  $\min_{\gamma \in F' \minus \identity} \delta(\gamma)$, we obtain
  the desired inclusions for every
  $\gamma \in \Gamma_i \minus \{I_d\}$.
\end{proof}

\begin{proof}[Proof of Theorem~\ref{thm:pingpong_word_bound}
  \ref{item:wordlength_bound}]
  Consider a reduced alternating word $\gamma_1 \cdots \gamma_n$ in
  $\langle \Gamma_1, \Gamma_2 \rangle$, with $\gamma_1 \in \Gamma_i$
  and $\gamma_n \in \Gamma_j$.

  For any $r > 0$ and any $U \subset \P(\K^d)$, let $U^r$ denote the
  $r$-neighborhood $N_r(U)$. Fix $\delta > 0$ from
  \Cref{lem:eps_nbhd_pingpong}. The previous
  \cref{lem:eps_nbhd_pingpong} implies that the
  $U_i^\delta, U_j^\delta$ are both proper domains, so we may let
  $D_i, D_j$ respectively denote the diameters of the sets
  $U_i^{\delta/2}, U_j^{\delta/2}$ with respect to the Caratheodory
  metrics $\hil_{U_i^\delta}, \hil_{U_j^\delta}$ from
  \Cref{sec:caratheodory_metric}.

  We may now inductively apply \Cref{lem:eps_nbhd_pingpong} and
  \Cref{thm:caratheodory_metric_shrinks}
  \ref{item:c_metric_strong_contract} to see that there is a uniform
  constant $\lambda > 1$ so that, with respect to the metric
  $U_i^\delta$, the diameter of the set $\gamma_1 \cdots \gamma_nU_j$
  is at most $\lambda^{-(n-1)}D_j$. Now, observe that the function
  $\hil_{U_i^\delta}:U_i \times U_i \to \Real_{\ge 0}$ is continuously
  differentiable. So, since $\overline{U_i}$ is a compact subset of
  $U_i^\delta$, the metrics $\hil_{U_i^\delta}$ and $\hil_\P$ are
  bi-Lipschitz equivalent on $U_i$, and thus the diameter of
  $\gamma_1 \cdots \gamma_nU_j$ (with respect to the metric $d_\P$) is
  at most $\lambda^{-n}A$ for a uniform constant $A < \infty$.

  Let $g = \gamma_1 \cdots \gamma_n$, and pick a Cartan decomposition
  $g = k_g \exp(\mu(g)) k_g'$. Recall that we have chosen
  $x_j \in U_j$ and $\theta > 0$ so that
  $B_{\theta}(x_j) \subset U_j$. Using \Cref{dist-hyp}, fix
  $x \in B_{\theta}(x_j)$ so that
  $\hil_\P(x, \P(k_g'^{-1}e_1^\perp)) > \theta/2$. Then apply
  Lemma~\ref{gap-estimate1} to the balls $B_{\theta/2}(x)$ and
  $B_r(gx)$ with $r = \lambda^{-n}A$ to obtain the desired estimate.
\end{proof}

The next step is to use the initial estimate above to prove the
following:
\begin{proposition}
  \label{prop:combination_divergent}
  The semigroup $\Delta = \langle \Gamma_1, \Gamma_2 \rangle$ is
  $1$-divergent. In particular, the set $F$ defined above is finite.
\end{proposition}
For a similar result to the above, see Theorem 3.1 in
\cite{Dey-Kapovich}. We provide a proof of this result since our
ping-pong setup is slightly different; in particular we do not assume
disjointness of the ping-pong sets $U_1, U_2$.

Note that we could also get a sharper version of
\Cref{prop:combination_divergent} as an immediate consequence of parts
\ref{item:wordlength_bound} and \ref{item:pingpong_12_bound} of
\Cref{thm:pingpong_word_bound}. However, we need the version above to
establish the rest of \Cref{thm:pingpong_word_bound}.
\begin{proof}[Proof of \Cref{prop:combination_divergent}]
  Suppose that the statement of the proposition is false. Then there
  is an infinite sequence $g_n$ of pairwise distinct elements in
  $\Delta$ such that $\frac{\sigma_1}{\sigma_2}(g_n)$ is bounded. We
  may write each of these elements as an alternating word
  \[
    g_n = \gamma_{1,n}\ldots \gamma_{m_n, n},
  \]
  so that each $\gamma_{i, n}$ is a nontrivial element in some
  $\Gamma_j$.

  Part \ref{item:wordlength_bound} of \Cref{thm:pingpong_word_bound}
  immediately implies that $m_n$ is uniformly bounded in $n$, so we
  may pass to a subsequence and assume that $m_n = m$ for some fixed
  $m$, independent of $n$. By passing to a further subsequence, we may
  assume that for every $n$, we have $\gamma_{1,n} \in \Gamma_j$ and
  $\gamma_{m,n} \in \Gamma_k$ for some fixed $j, k$; without loss of
  generality $j = 1$.

  To complete the proof, we use an iterative argument exploiting the
  relationship between $1$-divergence and expansion/contraction in
  $\P(\K^d)$; this is similar to the approach in
  e.g. \cite[Lem. 3.3]{Dey-Kapovich}. We first claim that the elements
  in the sequence $\gamma_{1,n}$ must lie in a fixed finite subset of
  $\Gamma_1$. If not, then after further extraction we may assume that
  $\gamma_{1,n} \in \Gamma_{1,n} \minus F$ for all $n$. Then by
  \Cref{estimate-Xi}, we have
  \[
    \dgr(\Xi_{d-1}\left(\gamma_{1,n}^{-1}), V_1\right) < \frac{\eps}{8},
  \]
  which implies that
  $\dist{\Xi_{d-1}(\gamma_{1,n}^{-1}), x} > \frac{7\eps}{8}$ for every
  $x \in U_2$. Then, by \Cref{lem:BPS-estimate-1}, since we assume
  $\Gamma_1$ is $1$-divergent, the diameter of the set
  $\gamma_{1,n}U_2$ tends to zero as $n \to \infty$. By the ping-pong
  setup, the diameter of
  \[
    g_nU_k = \gamma_{1,n} \ldots \gamma_{m,n}U_k
  \]
  also tends to zero. But, because of \Cref{gap-estimate1}, this
  contradicts the fact that $\frac{\sigma_1}{\sigma_2}(g_n)$ is
  bounded. This shows our claim, meaning that the elements appearing
  in the sequence $\gamma_{1,n}$ lie in some finite subset of
  $\Gamma_1$.

  Then, since the elements in $g_n$ are pairwise distinct, a
  subsequence of
  $g_n' := \gamma_{1,n}^{-1}g_n = \gamma_{2,n} \ldots \gamma_{m,n}$ is
  pairwise distinct; further, by \Cref{gap-product-1}, we also know
  that $\frac{\sigma_1}{\sigma_2}(\gamma_{2,n} \ldots \gamma_{m,n})$
  is bounded. Arguing iteratively as above, we eventually see that
  (after further extraction) each of the $m$ sequences
  $\gamma_{\ell,n}, 1 \le \ell \le m$ lies in a finite subset of some
  $\Gamma_j$, contradicting the fact that the $g_n$ are pairwise
  distinct.
\end{proof}

The proof of the rest of \Cref{thm:pingpong_word_bound} essentially
follows the same rough outline as the proof of
\Cref{prop:combination_divergent}, but with more precise estimates
throughout.

\begin{proof}[Proof of Theorem \ref{thm:pingpong_word_bound}
  \ref{item:pingpong_1_bound}] Via \Cref{prop:combination_divergent},
  we can define $C_0:=\max_{f\in F}\frac{\sigma_1}{\sigma_d}(f)$ and
  set
\[
C_2:=\max \left\{\frac{2}{\varepsilon},C_0^{2}\right\}.
\]

We shall prove, using induction on $n\in \mathbb{N}$, that for every reduced word $g_1\cdots g_n\in \Delta$, $g_i\in \Gamma_i\minus \{I_d\}$, we have that \begin{align}\label{ineq-bound1-induction}\sigma_1\left(g_1\cdots g_n\right)\geq C_2^{-n}\sigma_1(g_1)\cdots \sigma_1(g_n).\end{align}

Before we proceed with the induction, we observe the following. For
any $h_1,h_2 \notin F$ and any reduced word $h_i\in \Delta$ of the
form $h_i=h_{i1}\cdots h_{im_i}$, such that
$h_{1m_1}\in \Gamma_{i_1}$, $h_{2m_2}\in \Gamma_{i_2}$, $i_1\neq i_2$,
by Lemma \ref{estimate-Xi} (i) and (ii) and (\ref{est-1-}) we have
that
\begin{align}\label{dist-reducedword1}
  \dist{\Xi_1(h_2),\Xi_{d-1}(h_1^{-1})}
  &\geq \inf_{y \in V_{i_1}}\dist{\Xi_1(h_2), y}-\frac{\varepsilon}{8}\geq \dist{U_{i_2},
    V_{i_1}}-\frac{3\varepsilon}{8} >\frac{\varepsilon}{2}.
\end{align}

Note that the desired statement holds trivially when $n=1$, so assume
that (\ref{ineq-bound1-induction}) holds for reduced words in $n$
elements, and fix a reduced word $g_1\cdots g_n g_{n+1}\in \Delta$. If
either $g_1\cdots g_n \in F$ or $g_{n+1}\in F$, then note that
$$\sigma_1\left(g_1\cdots g_ng_{n+1}\right)\geq
C_0^{-1}\sigma_1(g_1\cdots g_n)\sigma_1(g_{n+1})\geq
C_2^{-n-1}\sigma_1(g_1)\cdots \sigma_1(g_{n+1})$$ and the statement
holds true. So, suppose that $g_1\cdots g_n\notin F$ and
$g_{n+1}\notin F$.

Since $g_1\cdots g_{n}g_{n+1}\in \Delta$ is reduced, by
(\ref{dist-reducedword1}) and Lemma \ref{ratio-sigma1} we have
that \begin{align*} \frac{\sigma_1(g_1\cdots
    g_ng_{n+1})}{\sigma_1(g_1\cdots g_{n})\sigma_1(g_{n+1})}&\geq
  \dist{\Xi_{1}(g_{n+1}), \Xi_{d-1}((g_1\cdots g_n)^{-1})}\geq
  \frac{\varepsilon}{2}.\end{align*} In particular, by the inductive
hypothesis, since $\frac{\varepsilon}{2}\geq C_2^{-1}$, we have
that
$$\sigma_1\left(g_1\cdots g_ng_{n+1}\right)\geq
C_2^{-n-1}\sigma_1(g_1)\cdots \sigma_1(g_{n+1}).$$ This completes the
proof of the induction and (\ref{ineq-bound1-induction})
follows.\end{proof}

\begin{proof}[Proof of Theorem \ref{thm:pingpong_word_bound}
  \ref{item:pingpong_12_bound}]
  Again due to \Cref{prop:combination_divergent}, we can define
  \[
    C_3:=\max \left\{\max_{f\in
        F}\Big(\frac{\sigma_1}{\sigma_d}(g)\Big)^2,\frac{2^{9}}{7\varepsilon^3}\right\}.
  \]
 We shall use induction to prove the following bound: every reduced
 word $g_1\cdots g_n\in \Delta$ satisfies the estimate
 \begin{align}\label{bound-12}
   \frac{\sigma_1}{\sigma_2}\left(g_1\cdots g_n\right)\geq C_3^{-n}
   \frac{\sigma_1}{\sigma_2}(g_1)\cdots
   \frac{\sigma_1}{\sigma_2}(g_n).
 \end{align}

  The bound holds trivially for $n=1$. Suppose now that we know that
  (\ref{bound-12}) holds for $n\in \mathbb{N}$. Fix a reduced word
  $g_1\ldots,g_{n}g_{n+1}\in \Delta$, such that
  $g_{n+1}\in \Gamma_{j_1}$, $g_n\in \Gamma_{j_2}$, $j_1\neq j_2$.  If
  either $g_{n+1}\in F$ or $g_1\cdots g_n\in F$, by Lemma
  \ref{gap-product-1}, we have
  that
  $$\frac{\sigma_1}{\sigma_2}(g_1\cdots g_ng_{n+1})\geq
  C_3^{-1}\frac{\sigma_1}{\sigma_2}(g_1\cdots
  g_n)\frac{\sigma_1}{\sigma_2}(g_{n+1})\geq
  C_3^{-n-1}\frac{\sigma_1}{\sigma_2}(g_1)\cdots
  \frac{\sigma_1}{\sigma_2}(g_{n+1})$$ and (\ref{bound-12}) holds for
  $g_1\cdots g_ng_{n+1}\in \Delta$.

Now suppose that $h_{n}:=g_1\cdots g_n\notin F$ and $g_{n+1}\notin
F$. Consider the ball $B_{\theta}(x_{j_2})\subset U_{j_2}$ and note
that by (\ref{est-1-}) and Lemma \ref{estimate-Xi} (ii) that
$\textup{dist}(x_{j_2},\Xi_{d-1}(g_{n+1}^{-1}))\geq
\frac{7\varepsilon}{8}.$ In particular, Lemma \ref{Lipschitz} implies
that $$g_{n+1}B_{\theta}(x_{j_2})\subset
B_{\epsilon_{n+1}}(g_{n+1}x_{j_2}),$$ where
$\epsilon_{n+1}:=\frac{16\theta}{7\varepsilon}\frac{\sigma_2}{\sigma_1}(g_{n+1})$. In
addition observe
that $$d_{\mathbb{P}}\left(g_{n+1}x_{j_2},U_{j_1}\right)\leq
d_{\mathbb{P}}\left(g_{n+1}x_{j_2},\Xi_1(g_{n+1})\right)+d_{\mathbb{P}}\left(\Xi_1(g_{n+1}),U_{j_1}\right)\leq
\frac{8}{7\varepsilon}\frac{\sigma_2}{\sigma_1}(g_{n+1})+\frac{\varepsilon}{8}
< \frac{3\varepsilon}{8}.$$

Since $g_{1}\cdots g_{n}g_{n+1}\in \Delta$ is reduced and
$h_n=g_1\cdots g_n\notin F$, by Lemma \ref{estimate-Xi} (ii) we have
that
\[
  \dist{g_{n+1}x_{j_2},\Xi_{d-1}(h_{n}^{-1})}\geq
  \dist{U_{j_1},V_{j_2}}-\frac{3\varepsilon}{8}-\frac{\varepsilon}{8}\geq
  \frac{\varepsilon}{2}.
\]
By applying Lemma \ref{Lipschitz} for the action of $h_{n}$ on the
ball $B_{\epsilon_{n+1}}(g_{n+1}x_{j_2})$ and the previous estimate,
we conclude that
\[
  h_{n}B_{\epsilon_{n+1}}(g_{n+1}x_{j_2})\subset
  B_{\epsilon_{n+1}'}(h_ng_{n+1}x_{j_2}), \ \textup{where}
\]

\[
\epsilon_{n+1}':=\frac{4}{\varepsilon}\frac{\sigma_2}{\sigma_1}(h_{n})\epsilon_{n+1}=\frac{64\theta}{7\varepsilon^2}\frac{\sigma_2}{\sigma_1}(g_{n+1})\frac{\sigma_2}{\sigma_1}(h_n).
\]

At this point note that we do not necessarily know that
$h_{n}g_{n+1}\notin F$, nor that $h_ng_{n+1}\in \Delta$ has a gap of
index $1$. However, we have assumed $g_{n+1}, h_n \notin F$, so if we
write $h_{n}g_{n+1}=k_n\exp(\mu(h_ng_{n+1}))k_n'$ for some
$k_n,k_n'\in \mathsf{K}_d$, by Lemma \ref{BPS-estimate-2} (ii), we have
that
\begin{align}\label{h_ng_{n+1}-bound}
  \dgr\left((k_n')^{-1}e_d^{\perp},
    \Xi_{d-1}(g_{n+1}^{-1})\right)\leq
  \frac{\sigma_2}{\sigma_1}(g_{n+1}) \cdot
  \frac{\sqrt{d-1}}{\dist{\Xi_1(g_{n+1}),\Xi_{d-1}(h_n^{-1})}}.
\end{align}

At this point, since $g_{n+1}\notin F$, by \eqref{lower-1}, we have
$\frac{\sigma_1}{\sigma_2}(g_{n+1})>\frac{8\sqrt{d-1}}{\varepsilon^2}$. In
addition, by \Cref{estimate-Xi} we know that $\dgr\left(
    \Xi_{d-1}(h_{n}^{-1}),V_{j_2}\right)\leq \frac{\varepsilon}{8}$, hence $$\textup{dist}\left(\Xi_1(g_{n+1}),
    \Xi_{d-1}(h_{n}^{-1})\right)\geq \textup{dist}\left(U_{j_1},
    V_{j_2}\right)-\frac{\varepsilon}{8}-\frac{\varepsilon}{8}\geq \frac{3\varepsilon}{4}.$$
The previous bound and (\ref{h_ng_{n+1}-bound}) imply the bound
\[
  \dgr\left((k_n')^{-1}e_d^{\perp},
    V_{j_1}\right)\leq \dgr\left((k_n')^{-1}e_d^{\perp},
    \Xi_{d-1}(g_{n+1}^{-1})\right)+ \dgr\left(
    \Xi_{d-1}(g_{n+1}^{-1}), V_{j_1}\right)<\frac{\varepsilon}{8}+ \frac{\eps}{4}.
\]
In particular, since $x_{j_2}\in U_{j_2}$
$$\textup{dist}(x_{j_2},\mathbb{P}((k_n')^{-1}e_d^{\perp}))\geq
\textup{dist}\left(U_{j_2},V_{j_1}\right)-\frac{3\varepsilon}{8}>
\frac{\varepsilon}{2}.$$

Then, since $g_1\cdots g_ng_{n+1}B_{\theta}(x_{j_2})\subset B_{\epsilon_{n+1}'}(g_1\cdots g_{n+1}x_{j_2})$, by the previous estimate, Lemma \ref{gap-estimate1} and the inductive step, we conclude that \begin{align*} \frac{\sigma_1}{\sigma_2}(g_1\cdots g_ng_{n+1})&\geq \frac{\theta}{4\epsilon_{n+1}'}\textup{dist}(x_{j_2},\mathbb{P}((k_n')^{-1}e_d^{\perp}))\geq \frac{7\varepsilon^3}{2^{9}}\frac{\sigma_1}{\sigma_2}(g_1\cdots g_n)\frac{\sigma_1}{\sigma_2}(g_{n+1})\\  &\geq C_3^{-(n+1)}\frac{\sigma_1}{\sigma_2}(g_1)\cdots \frac{\sigma_1}{\sigma_2}(g_{n+1}).\end{align*} This concludes the proof of part \ref{item:pingpong_12_bound}.
\end{proof}

\begin{proof}[Proof of Theorem \ref{thm:pingpong_word_bound} \ref{item:pingpong_12_bound-eigen}]
Let $n\in \mathbb{N}$ be an even integer and a reduced word $\gamma_1\gamma_2\cdots \gamma_n\in \langle \Gamma_1,\Gamma_2\rangle$. Since $\gamma_1,\gamma_n\in \Gamma_1\cup \Gamma_2$ do not lie in the same semigroup, for every $m\in \mathbb{N}$ the word $(\gamma_1\cdots \gamma_n)^m\in \langle \Gamma_1,\Gamma_2\rangle$ is reduced and by Theorem \ref{thm:pingpong_word_bound} \ref{item:pingpong_1_bound} we have that 
\[
\sigma_1\left((\gamma_1\gamma_2\cdots \gamma_n)^m\right)\geq C_2^{-nm}\sigma_1(\gamma_1)^m\sigma_1(\gamma_2)^m\cdots \sigma_1(\gamma_n)^m.
\]
 Therefore, since $\ell_1(\gamma_1\gamma_2\cdots \gamma_n)=\lim_{m\rightarrow \infty}\sigma_1\left((\gamma_1\gamma_2\cdots \gamma_n)^m\right)^{\frac{1}{m}}$, we obtain the estimate: \begin{equation}\label{eigenvalue-estimate1} \ell_1\left(\gamma_1\gamma_2\cdots \gamma_n\right)\geq C_2^{-n} \sigma_1(\gamma_1)\sigma_1(\gamma_2)\cdots \sigma_1(\gamma_n).\end{equation} Using (\ref{eigenvalue-estimate1}) we finally obtain: 

\begin{align*}
  \frac{\ell_1}{\ell_2}\left(\gamma_1\gamma_2\cdots
  \gamma_n\right)
  &=\frac{\ell_1(\gamma_1\gamma_2\cdots
    \gamma_n)^2}{\ell_1(\wedge^2(\gamma_1\gamma_2\cdots
    \gamma_n))}\geq
    C_2^{-2n}\frac{\sigma_1(\gamma_1)^2\sigma_1(\gamma_2)^2\cdots
    \sigma_1(\gamma_n)^2}{\sigma_1(\wedge^2
    \gamma_1\gamma_2\cdots \gamma_n)}\\
  &\geq C_2^{-2n} \frac{\sigma_1(\gamma_1)^2\sigma_1(\gamma_2)^2\cdots
    \sigma_1(\gamma_n)^2}{\sigma_1(\wedge^2
    \gamma_1)\sigma_1(\wedge^2\gamma_2)\cdots
    \sigma_1(\wedge^2\gamma_n)}\\
  &=C_2^{-2n}\frac{\sigma_1}{\sigma_2}(\gamma_1)\frac{\sigma_1}{\sigma_2}(\gamma_2)\cdots
    \frac{\sigma_1}{\sigma_2}(\gamma_n).
\end{align*}
\end{proof}

\subsection{Consequences of \Cref{thm:pingpong_word_bound}}

Below we prove \Cref{cor:pingpong_qi_embed} and
\Cref{semigroup-Anosov}, and then use \Cref{semigroup-Anosov} to
construct an interesting example of a $1$-Anosov semigroup mentioned
in the introduction. We first make a basic observation:
\begin{proposition}
  \label{prop:free_product_faithful}
  In the context of \Cref{thm:pingpong_word_bound}, if any of the
  conditions below hold, then the group
  $\langle \Gamma_1, \Gamma_2 \rangle$ is naturally isomorphic to the
  free product $\Gamma_1 \ast \Gamma_2$:
  \begin{enumerate}[label=(\alph*)]
  \item\label{item:pingpong_forward} the ping-pong sets $U_1, U_2$ are
    disjoint,
  \item\label{item:pingpong_backward} the ping-pong sets $V_1, V_2$
    are disjoint,
  \item\label{item:pingpong_groups} $\Gamma_1$ and $\Gamma_2$ are both
    groups.
  \end{enumerate}
\end{proposition}
\begin{proof}
  If either \ref{item:pingpong_forward} or
  \ref{item:pingpong_backward} holds, then the result is an immediate
  consequence of the ping-pong lemma for semigroups. If
  \ref{item:pingpong_groups} holds, then combining
  \ref{item:wordlength_bound} and \ref{item:pingpong_12_bound} in
  \Cref{thm:pingpong_word_bound} we see that the kernel of the map
  $\rho:\Gamma_1 \ast \Gamma_2 \to \GL_d(\K)$ is a finite normal
  subgroup of the free product $\Gamma_1 \ast \Gamma_2$. Since
  $\Gamma_1$ and $\Gamma_2$ are both included into $\GL_d(\K)$, the
  kernel of $\rho$ is trivial and
  $\Gamma_1 \ast \Gamma_2 \to \langle \Gamma_1, \Gamma_2 \rangle$ is
  an isomorphism.
\end{proof}

\begin{remark}
  The conclusion of \Cref{prop:free_product_faithful} can fail if all
  three of \ref{item:pingpong_forward}, \ref{item:pingpong_backward},
  and \ref{item:pingpong_groups} do not hold. For a simple example,
  consider the semigroup $\Gamma$ generated by a nontrivial loxodromic
  element $g$ in $\SL(2, \Real)$, and take
  $\Gamma_1 = \Gamma_2 = \Gamma$. Then, if $U_1 = U_2 = U$ is a small
  neighborhood of the attracting point of $g$ in $\P(\Real^2)$, and
  $V_1 = V_2 = V$ is a small neighborhood of the repelling point in
  $\Gr_1(\Real^2) = \P(\Real^2)$, the ping-pong conditions are
  satisfied, but clearly $\Gamma_1 \ast \Gamma_2$ is not isomorphic to
  $\langle \Gamma_1, \Gamma_2 \rangle = \Gamma$.
\end{remark}

\begin{proof}[Proof of Corollary~\ref{cor:pingpong_qi_embed}]
  % For the first part of the corollary, it suffices to show that if
  % $w_k$ is any sequence of reduced words in
  % $\langle \Gamma_1, \Gamma_2 \rangle$, with the $w_n$ pairwise
  % distinct (as reduced words), then $\sigma_1(w_n)/\sigma_2(w_n)$
  % tends to infinity.

  % For any reduced word $w \in \langle \Gamma_1, \Gamma_2 \rangle$, we
  % define the \emph{relative length} $||w||_{1,2}$ by
  % \[
  %   ||\gamma_1 \cdots \gamma_n||_{1,2} = n.
  % \]
  % Note that, in the case our sequence of pairwise distinct reduced
  % words $w_k$ has relative length tending to infinity,
  % Theorem~\ref{thm:pingpong_word_bound} \ref{item:wordlength_bound}
  % immediately implies that $\sigma_1(w_k)/\sigma_2(w_k)$ tends to
  % infinity. On the other hand, if the relative lengths $||w_k||_{1,2}$
  % are all bounded by some fixed $N < \infty$, then up to subsequence
  % each reduced word has the form
  % \[
  %   \gamma_{1,k} \cdots \gamma_{N,k}.
  % \]
  % Up to further subsequence there is some fixed index $j$ so that all
  % of the elements $\gamma_{j,k}$ are pairwise distinct in either
  % $\Gamma_1$ or $\Gamma_2$. Since $\Gamma_1$ and $\Gamma_2$ are
  % divergent, it follows that
  % $\sigma_1(\gamma_{j,k}) / \sigma_2(\gamma_{j,k})$ tends to
  % infinity. Thus by \Cref{thm:pingpong_word_bound} part
  % \ref{item:pingpong_12_bound} we have
  % \[
  %   \frac{\sigma_1}{\sigma_2}(w_k) \ge c_1^N
  %   \frac{\sigma_1}{\sigma_2}(\gamma_{j,k}) \to \infty.
  % \]
  % We conclude that every subsequence of $w_k$ has a further
  % subsequence which is $1$-divergent, meaning that $w_k$ is itself
  % $1$-divergent.
  Fix finite generating sets for $\Gamma_1, \Gamma_2$, which induce a
  word metric $|\cdot|$ on $\Gamma_1, \Gamma_2$, and the abstract free
  product $\Gamma_1 * \Gamma_2$.  We have assumed that
  $\Gamma_1, \Gamma_2$ are quasi-isometrically embedded, which means
  that there is a uniform constant $L > 0$ such that
  \begin{equation}
    \label{eq:qi_assumption}
    \sigma_1(\gamma_i) \ge L |\gamma_i|
  \end{equation}
  for every $\gamma_i \in \Gamma_1 \cup \Gamma_2$.

  Let $\gamma_1 \cdots \gamma_n$ be a reduced word in
  $\langle \Gamma_1, \Gamma_2 \rangle$. By
  \Cref{thm:pingpong_word_bound} \ref{item:pingpong_1_bound}, for a
  uniform constant $c_1 > 0$ we also have the inequality
  \begin{equation}
    \label{eq:s1_inequality}
    \sigma_1(\gamma_1 \cdots \gamma_n) \ge c_1^n \sigma_1(\gamma_1)
    \cdots \sigma_1(\gamma_n).
  \end{equation}
  Without loss of generality we can assume that $c_1 < 1$, so
  $\log c_1 < 0$.

  Let $g = \gamma_1 \cdots \gamma_n$, and let $\ell = |g| / n$. We
  consider two cases.
  \begin{description}
  \item[Case 1: $\ell \ge \frac{-2}{L}\log c_1$] In this case
    \eqref{eq:qi_assumption} and \eqref{eq:s1_inequality} give us
    \begin{align*}
      \log(\sigma_1(g))
      &\ge n\log(c_1) + \sum_{i=1}^n L|\gamma_i|\\
      &= \frac{|g|}{\ell}\log(c_1) + L|g|\\
      &= |g|\Big(\frac{1}{\ell}\log(c_1) + L\Big)\\
      &\ge |g|\frac{L}{2}.
    \end{align*}
  \item[Case 2: $\ell \le \frac{-2}{L}\log c_1$] In this case, setting
    $L' = \frac{L}{-2\log c_1} > 0$, we have $1/\ell \ge L'$ and thus
    \[
      n = |g|/\ell \ge L'|g|.
    \]
    Then \Cref{thm:pingpong_word_bound} \ref{item:wordlength_bound}
    implies that for some uniform $\lambda > 1$ and $c_4 > 0$ we have
    \[
      \log\frac{\sigma_1}{\sigma_2}(g) \ge \log(\lambda)L'c_4|g|.
    \] 
    
    %This shows that the induced representation $\Gamma_1\ast \Gamma_2\rightarrow \mathsf{GL}_d(\K)$ (restricting to the identity map on $\Gamma_i$) is a quasi-isometric embedding. In particular, since its kernel if finite and is faithful on $\Gamma_i$, it follows that $\langle \Gamma_1,\Gamma_2\rangle \cong \Gamma_1\ast \Gamma_2$
  \end{description}
\end{proof}

\begin{proof}[Proof of Corollary \ref{semigroup-Anosov}] By assumption, there is $V_0\in \mathsf{Gr}_{d-1}(\mathbb{K}^d)$ with $\Lambda_1(\Gamma)\subset \mathbb{P}(\mathbb{K}^d)\minus \mathbb{P}(V_0)$. Since the limit set of the dual group $\Gamma^{\ast}$ is contained in an affine chart of $\mathbb{P}(\mathbb{K}^d)$, the complement of $\bigcup_{W\in \Lambda_{d-1}(\Gamma)}\mathbb{P}(W)$ is a non-empty closed subset of $\mathbb{P}(\mathbb{K}^d)$. In particular, we may choose $[v_0]\in \mathbb{P}(\mathbb{K}^d)\minus \mathbb{P}(V_0)$ and $\varepsilon>0$ with 
\[
\textup{dist}\left([v_0],\mathbb{P}(W)\right)\geq 6\varepsilon, \ \forall \ W\in \Lambda_{d-1}(\Gamma).
\]

Now let $g\in \mathsf{GL}_d(\mathbb{K})$ be a $1$-proximal matrix with
attracting fixed point in $[v_0]\in \mathbb{P}(\mathbb{K}^d)$ and
repelling hyperplane $V_0\in \mathsf{Gr}_{d-1}(\mathbb{K})$. By the
previous choices, we may pass to a finite-index subgroup
$\Gamma'<\Gamma$ and choose $r>0$ large enough such that for every
$n\in \mathbb{N}$,
\begin{align*} \left(\Gamma' \minus
  \{I_d\}\right)B_{\varepsilon}([v_0])\subset
  N_{\varepsilon}(\Lambda_{1}(\Gamma)),
  &\ g^{rn}N_{\varepsilon}\left(\Lambda_1(\Gamma)\right)\subset
    B_{\varepsilon}([v_0])\\
  \left(\Gamma' \minus
  \{I_d\}\right)N_{\varepsilon}(\mathbb{P}(V_0))\subset
  N_{\varepsilon}(\Lambda_{d-1}(\Gamma)),
  &\ g^{-rn}N_{\varepsilon}(\Lambda_{d-1}(\Gamma))\subset
    N_{\varepsilon}(\mathbb{P}(V_0)).
\end{align*}
By the previous inclusions we deduce that the semigroups
$\Gamma'<\Gamma$ and $\{g^{rm}:m\geq 0\}$ are in ping-pong position in
the sense of Definition \ref{pingpong-position}; moreover, the
relevant ping-pong sets are also disjoint by construction. Then by
\Cref{prop:free_product_faithful} the semigroup
$\langle \Gamma', g^r\rangle < \mathsf{GL}_d(\mathbb{K})$ is
isomorphic to the free product of $\Gamma'$ and
$H = \{g^{rm}:m\geq 0\}$. Then passing to a further finite-index
subgroup $\Gamma'' < \Gamma'$, and increasing $r$ if necessary, we can
use Theorem \ref{thm:pingpong_word_bound} \ref{item:pingpong_12_bound}
to conclude that the semigroup is $1$-Anosov (alternatively, one may
also apply \Cref{thm:pingpong_word_bound} \ref{item:wordlength_bound}
as in the proof of \Cref{cor:qi_free_products} to see that the free
product of $\Gamma'$ and $H$ is already $1$-Anosov).
\end{proof}

\subsubsection{A $1$-Anosov semigroup which does not generate a
  $1$-Anosov subgroup}
\label{sec:anosov_semigroup_not_subgroup}

Let $\Gamma < \GL_d(\Real)$ be a discrete Gromov-hyperbolic group
which acts properly discontinuously and cocompactly on a properly
convex domain $\Omega \subset \P(\Real^d)$ (see
\Cref{sec:properly_convex_domains} below). Then the inclusion
$\Gamma \hookrightarrow \GL_d(\Real)$ is $1$-Anosov by
\cite[Prop. 6.1]{GW}. By \Cref{semigroup-Anosov}, we can replace
$\Gamma$ with some finite-index subgroup of itself and find some
$w \in \GL_d(\Real)$ so that the semigroup
$\langle \Gamma, w \rangle \subset \GL_d(\Real)$ is $1$-Anosov and
isomorphic to the free product of $\Gamma$ with the semigroup
$\{w^n:n\geq 0\}$. However, the group generated by
$\Gamma \cup\{w^{\pm1}\}$ fails to be $1$-Anosov, since $\Gamma$
cannot be embedded as an infinite index subgroup of any $1$-Anosov
subgroup of $\mathsf{GL}_d(\mathbb{R})$ (see \cite[Thm. 1.5 \&
1.7]{CT}).

\section{Free products of quasi-isometrically embedded discrete groups}

In this section we prove Theorem \ref{discrete-freeprod-1} and its
consequence Corollary \ref{cor:qi_free_products}.

\subsection{Properly convex
  domains}\label{sec:properly_convex_domains}

Our proof of \Cref{discrete-freeprod-1} requires us to work with
\emph{properly convex domains} in real projective space. Recall that
an open subset $\Omega \subset \P(\Real^d)$ is a properly convex
domain if its closure is a bounded convex subset of some affine chart
in $\P(\Real^d)$. Equivalently, $\Omega$ is the projectivization of a
\emph{properly convex cone} $\tilde{\Omega}$ in $\Real^d$, i.e. a
convex subset of $\Real^d$ invariant under multiplication by positive
scalars whose closure does not contain any $k$-plane in $\Real^d$ for
$k > 0$.

If $\Omega$ is a properly convex domain, then the
set of elements $g \in \GL_d(\Real)$ preserving $\Omega$ is called the
\emph{automorphism group} and denoted $\Aut(\Omega)$. When viewed as a
subgroup of $\PGL_d(\Real)$ by projectivization, $\Aut(\Omega)$ acts
properly on $\Omega$.

We need the following basic result on sequences in $\Aut(\Omega)$:
\begin{lemma}
  \label{lem:xi_avoids_compact}
  Let $K$ be a compact subset of a properly convex domain
  $\Omega \subset \P(\Real^d)$. There exists $\eps > 0$ and $L > 0$ so
  that whenever $g \in \Aut(\Omega)$ satisfies
  $\frac{\sigma_k}{\sigma_{k+1}}(g) > L$, then
  \[
    \inf_{x \in K}\dist{x, \P(\Xi_k(g))} > \eps.
  \]
\end{lemma}
\begin{proof}
  Suppose the statement is false. Then there is an index $k$ and a
  sequence $g_n \in \Aut(\Omega)$ with
  $\frac{\sigma_k}{\sigma_{k+1}}(g_n) \to \infty$, satisfying
  \[
    \inf_{x \in K}\dist{x, \P(\Xi_k(g_n))} \to 0.
  \]
  After extracting a subsequence, $(\Xi_k(g_n))_{n\in \mathbb{N}}$ converges to some
  $k$-plane $V_k$, with $\P(V_k) \cap K \ne \emptyset$. After further
  extraction, $(\Xi_{d-k}(g_n^{-1}))_{n\in \mathbb{N}}$ also converges to some
  $(d-k)$-plane $V_{d-k}$.

  Now, if $V_k'$ is any $k$-plane transverse to $V_{d-k}$, we must
  have $g_n\P(V_k') \to \P(V_k)$ with respect to Hausdorff distance on
  $\P(\Real^d)$. Since $\Omega$ is properly convex, we can choose such
  a $V_k'$ so that
  $\P(V_k') \subset \P(\Real^d) \minus \overline{\Omega}$.

  Let $r > 0$ be such that $N_r(K) \subset \Omega$. Since $\P(V_k)$
  intersects $K$, for all sufficiently large $n$ there is some
  $x_n \in \P(V_k')$ with $g_nx_n \in N_r(K)$, hence
  $g_nx_n \in \Omega$. But since $x_n \notin \Omega$ this contradicts
  the fact that $g_n \in \Aut(\Omega)$.
\end{proof}

\subsection{Establishing the theorem}

The proposition below is the key step in the proof of
\Cref{discrete-freeprod-1}.
\begin{proposition}
  \label{prop:pingpong_convex_norm_bound}
  Let $\eta \in (0, 1)$, let $\Omega$ be a properly convex domain in
  $\P(\Real^d)$, and let $\Gamma$ be a subgroup of
  $\Aut(\Omega) < \GL_d(\Real)$ satisfying
  $\sigma_1(g)\sigma_d(g) = 1$ for all $g \in \Gamma$. There exists:
  \begin{itemize}
  \item a compact neighborhood $\mathcal{K} \subset \GL_d(\Real)$ of
    the identity;
  \item a pair of inclusions
    $\rho_1, \rho_2:\Gamma \hookrightarrow \GL_{d+1}(\Real)$, and
  \item a pair of subsets
    $\mathcal{C}_1, \mathcal{C}_2 \subset \P(\Real^{d+1})$ with
    $\overline{\mc{C}_1} \cap \overline{\mc{C}_2} = \emptyset$
  \end{itemize}
  such that, for any $g \in \Gamma \minus \mathcal{K}$, the following
  two conditions hold:
  \begin{enumerate}[label=(\alph*)]
  \item\label{item:pingpong_convex_sets} for $i = 1,2$, we have
    $\rho_i(g)\mathcal{C}_{3-i} \subset \mathcal{C}_i$, and
  \item\label{item:convex_set_norm_estimate} for $i = 1, 2$ and any
    $[v] \in \mc{C}_{3-i}$, we have
    \[
      ||\rho_i(g)v|| \ge \sigma_1(g)^{1 - \eta}||v||.
    \]
  \end{enumerate}
\end{proposition}

To prove the proposition we need one more technical estimate.
\begin{lemma}\label{freeproduct-estimate-1} Let $g\in \GL_d(\K)$ satisfy
  $\sigma_d(g)=\sigma_1(g)^{-1}$, $d\geq 4$ and
  $0<\epsilon<\frac{2}{d-1}$. There exists $k\in \{1,\ldots, d-1\}$
  such that
  \[
    \sigma_k(g)\geq \sigma_1(g)^{1-d\epsilon}, \
    \frac{\sigma_k}{\sigma_{k+1}}(g)\geq \sigma_1(g)^{\epsilon}.
  \]
\end{lemma}

\begin{proof} Note that
  $\prod_{i=1}^{d-1}\frac{\sigma_i}{\sigma_{i+1}}(g)=\sigma_1(g)^2$. Consequently,
  there is at least one index $j \in \{1,\ldots,d-1\}$ such that
  \[
    \frac{\sigma_{j}}{\sigma_{j+1}}(g)\geq
    \sigma_1(g)^{\frac{2}{p-1}}\geq \sigma_1(g)^{\epsilon}.
  \]
  
  Let $k$ be
  a minimal such $j$. If $k=1$, the lemma follows. If $k>1$, by the minimality of $k$,
  for $1\leq i< k$ we have
  $\frac{\sigma_{i}}{\sigma_{i+1}}{(g)}\leq \sigma_1(g)^{\epsilon}$
  and therefore $\sigma_i(g)\geq \sigma_1(g)^{1-(i-1)\epsilon}$. The
  lemma follows.
\end{proof}

\begin{proof}[Proof of \Cref{prop:pingpong_convex_norm_bound}]
  Identify $\Real^d$ with a subspace $V \subset \Real^{d+1}$, and let
  $U$ be an orthogonal complement to $V$ in $\Real^{d+1}$. Then
  $\GL_d(\Real)$ is identified with $\GL(V)$. We implicitly extend the
  action of $\GL(V)$ to an action on $\Real^{d+1}$ by letting $\GL(V)$
  act by the identity on $U$, and define the representation $\rho_1$
  to be the restriction of this action to $\Gamma < \GL(V)$. We define
  the set $\mc{C}_1$ by taking it to be an open neighborhood of
  $\Omega$ in $\P(\Real^{d+1})$, chosen to be small enough so that the
  closure of $\mc{C}_1$ is contained in an affine chart, and does not
  contain the point $\P(U)$.

  The representation $\rho_2$ will be a conjugate of $\rho_1$ by some
  fixed element $w \in \GL_{d+1}(\Real)$. To find $w$, first let
  $\tilde{\Omega} \subset V$ be a convex cone lifting $\Omega$.
  Observe that the sum $\tilde{\Omega} + U$ is an open (non-properly)
  convex cone $\tilde{C} \subset \Real^{d+1}$. The projectivization
  $C$ of this cone contains a copy of $\Omega$, viewed as a subset of
  $\P(V) \subset \P(\Real^{d+1})$.

  Note that since $\overline{\mc{C}_1}$ does not contain $\P(U)$, the
  set $C \minus \overline{\mc{C}_1}$ is nonempty and open. So, we can
  let $\mc{C}_2$ be an open ball in $\P(\Real^{d+1})$ such that
  $\overline{C}_2 \subset C \minus \overline{\mc{C}_1}$ Then, we can
  choose a $\{1,d\}$-proximal element $w \in \GL_{d+1}(\Real)$ whose
  attracting and repelling fixed points both lie in $\mc{C}_2$, and
  whose attracting and repelling hyperplanes both avoid
  $\overline{\mc{C}_1}$. After replacing $w$ with a sufficiently large
  power $w^n$, we can ensure that
  \begin{equation}
    \label{eq:proximal_ball_inclusion}
    w\overline{\mc{C}_1} \cup w^{-1}\overline{\mc{C}_1} \subset \mc{C}_2.
  \end{equation}
  With this choice of $w$, we define $\rho_2(g) := w\rho_1(g)w^{-1}$.
  
  It remains to show that we can choose a compact neighborhood
  $\mc{K}$ of the identity so that conditions
  \ref{item:pingpong_convex_sets} and
  \ref{item:convex_set_norm_estimate} above hold for all
  $g \in \Gamma \minus \mc{K}$. To show
  \ref{item:pingpong_convex_sets}, we employ a strategy of
  Danciger-Gu\'eritaud-Kassel: by \cite[Prop. 4.3]{DGK24}, there is a
  family $\{\Omega_t : t \ge 0\}$ of $\rho_1(\Gamma)$-invariant
  properly convex open subsets of $\P(\Real^{d+1})$ such that
  $\Omega_t \subset \Omega_{t'}$ for all $t < t'$ and
  \[
    \bigcup_{t \ge 0} \Omega_t = C, \ \ \bigcap_{t \ge 0} \Omega_t
    = \Omega.
  \]
  Since $\overline{\mc{C}_2} \subset C$, we must have
  $\overline{\mc{C}_2} \subset \Omega_t$ for some $t <
  \infty$. Moreover, \cite[Prop. 4.2]{DGK24} implies that for every
  $t \ge 0$, the \emph{full orbital limit set} of $\rho_1(\Gamma)$ in
  $\Omega_t$ does not depend on $t$, and therefore is a subset of
  $\overline{\Omega}$. (When a group $G$ acts on a properly convex
  domain $\Omega$, the full orbital limit set of $G$ in $\Omega$ is
  the set of all accumulation points in $\partial\Omega$ of $G$-orbits
  in $\Omega$.)

  Since $\Omega_t$ is properly convex, the $\rho_1$-action of $\Gamma$
  on $\Omega_t$ is proper. So, there is a compact neighborhood
  $\mc{K}$ of the identity so that for every
  $g \in \Gamma \minus \mc{K}$, we have
  \[
    \rho_1(g)\mc{C}_2 \subset \mc{C}_1.
  \]
  This implies that \ref{item:pingpong_convex_sets} holds when
  $i = 1$. The corresponding inclusion when $i = 2$ then follows from
  \eqref{eq:proximal_ball_inclusion} and the definition of $\rho_2$.

  Finally we turn to the proof that
  \ref{item:convex_set_norm_estimate} holds for some choice of
  $\mc{K}$. Note that for any $g \in \Gamma$, if $\rho_i(g)$ has a gap
  of index $k$, and we write in the Cartan decomposition $$\rho_i(g)=k_{g,i}\textup{diag}\big(\sigma_1(\rho_{1}(g)),\ldots, \sigma_{d+1}(\rho_i(g))\big)k_{g,i}', \ k_{g,i},k_{g,i}'\in \mathsf{K}_d$$ then for any $v \in \Real^{d+1}$, we have
  \begin{align}
    \label{eq:rho1_norm_inequality}
   \nonumber \big|\big|\rho_i(g)v\big|\big|&=\big|\big| \textup{diag}\big(\sigma_1(\rho_{i}(g)),\ldots, \sigma_{d+1}(\rho_i(g))\big)k_{g,i}'v\big|\big| \ge \left(\sum_{j=1}^{k}\sigma_{j}(\rho_j(g))^2 \big|\langle k_{g,i}'v,e_j\rangle \big|^2\right)^{1/2}\\
   \nonumber & \ge \sigma_k(\rho_i(g)) \left(\sum_{j=1}^{k}\big|\langle v, (k_{g,i}')^{-1}e_j\rangle \big|^2\right)^{1/2}\\
   & = \sigma_k(\rho_i(g)) \dist{[v], \P(\Xi_{d+1 -
        k}(g^{-1}))}\cdot ||v||.
  \end{align}

  Observe that, since we assume $\sigma_1(g)\sigma_d(g) = 1$ for any
  $g \in \Gamma$, the diagonal embedding
  $\rho_1:\Gamma \to \GL_{d+1}(\Real)$ satisfies
  $\sigma_1(g) = \sigma_1(\rho_1(g))$ and
  $\sigma_d(g) = \sigma_{d+1}(\rho_1(g))$. So, by applying
  \Cref{freeproduct-estimate-1} with $\epsilon = \eta / 2d$, for any
  $g \in \Gamma$, we can find some index $k$ (depending on $g$) so
  that
  \begin{equation}
    \label{eq:rho_1_lowerbounds}
    \sigma_k(\rho_1(g)) \ge \sigma_1(g)^{1 -\frac {\eta}{2}}, \ 
    \frac{\sigma_k}{\sigma_{k+1}}(\rho_1(g)) \ge
    \sigma_1(g)^{\frac{\eta}{2d}}.
  \end{equation}
  We also know that for every $g \in \Gamma$, $\rho_1(g)$ preserves
  the properly convex domain $\Omega_t$ containing the compact set
  $\overline{\mc{C}_2}$. So, by using \Cref{lem:xi_avoids_compact}
  together with the second inequality in \eqref{eq:rho_1_lowerbounds},
  we can see that for some $\eps > 0$ and some compact neighborhood
  $\mc{K}$ of the identity, any $g \in \Gamma \minus \mc{K}$ has a gap
  of some index $k$, and
  \[
    \inf_{x \in \mc{C}_2}\dist{x, \Xi_{d + 1 - k}(g^{-1})} \ge \eps.
  \]
  Combining this with \eqref{eq:rho1_norm_inequality} and the first
  inequality in \eqref{eq:rho_1_lowerbounds}, we find that for any
  $[v] \in \mc{C}_2$ and $g \in \Gamma \minus \mc{K}$, we have
  \[
    ||\rho_1(g)v|| \ge \eps \cdot \sigma_1(g)^{1 - \frac{\eta}{2}}||v||.
  \]
  By further increasing $\mc{K}$ we can also ensure that
  $\eps > \sigma_1(g)^{-\eta/2}$ for any $g \in \Gamma \minus \mc{K}$,
  which completes the proof of part
  \ref{item:convex_set_norm_estimate} when $i = 1$. The case $i = 2$
  also follows: since $\rho_2$ is a conjugate of $\rho_1$ by some
  element $w$, and $w^{-1}\mc{C}_1 \subset \mc{C}_2$, we can increase
  $\mc{K}$ further and use the $i = 1$ case to see that for any
  $[v] \in \mc{C}_1$ and any $g \in \Gamma \minus \mc{K}$ we have
  \[
    ||w^{-1}\rho_2(g) w (w^{-1}v)|| \ge \sigma_1(g)^{1 -
      \frac{\eta}{2}}||w^{-1}v||.
  \]
  This implies that for some constant $A > 0$ depending only on $w$,
  we have
  \[
    ||\rho_2(g)v|| \ge A \cdot \sigma_1(g)^{1 - \frac{\eta}{2}}||v||.
  \]
  Then by increasing $\mc{K}$ even further we get the desired bound in
  this case as well.
\end{proof}

Now we can prove the main result of this section.
\begin{proof}[Proof of \Cref{discrete-freeprod-1}]
  Consider the representations
  \begin{align*}
    \Sym_d&:\SL_{d}(\Real) \to \GL(V), \ \Sym_{d}(g)X=g^tXg  \
            X\in V,\\
    \psi_{d}&:\GL(V) \to \GL(V\otimes V), \ \psi_d(h)=h\otimes
              h^{\ast} \ h\in \GL(V)
  \end{align*}
  where $V=\Sym_{d}(\Real)$ is the vector space of symmetric
  $d\times d$ matrices. Note that both the image of $\Sym_{d}$ and its
  dual preserve the properly convex domain $D\subset \mathbb{P}(V)$ of
  positive definite symmetric matrices. The composition
  $\phi_d = \psi_d\circ \Sym_{d}$ preserves the properly convex domain
  $D':=\big\{[X\otimes Y]: [X],[Y]\in D\big\}$ of
  $\mathbb{P}(V\otimes V)$. Also, for any $g \in \SL_d(\Real)$, we
  have
  \[
    \sigma_1(\phi_d(g^{\pm 1})) = \frac{\sigma_1(g)^2}{\sigma_d(g)^2}.
  \]

  We are now in a position to apply
  \Cref{prop:pingpong_convex_norm_bound}. Set $\eta = \epsilon/2$, and
  let $\rho:\SL_d(\Real) \ast \SL_d(\Real) \to \GL_m(\Real)$ be the
  representation obtained by composing $\phi_d$ with $\rho_1$ on the
  first $\SL_d(\Real)$ factor and $\rho_2$ on the second factor. It is
  immediate from part \ref{item:pingpong_convex_sets} of the
  proposition that if $H_1, H_2 < \SL_d(\Real)$ are discrete subgroups
  with $\phi_d(H_i) \cap \mc{K} = \{\identity\}$, then the restriction
  $\rho:H_1 \ast H_2 \to \GL_m(\Real)$ is discrete and
  faithful. Moreover, if $\gamma_1 \cdots \gamma_n \in H_1 \ast H_2$
  is a reduced word, we can inductively apply part
  \ref{item:convex_set_norm_estimate} of the proposition to see that
  \[
    \frac{\sigma_1}{\sigma_m}\left(\rho(\gamma_1 \cdots
      \gamma_n)\right) \ge \prod_{i=1}^n \sigma_1(\phi_d(\gamma_i))^{1
      - \eta} = \prod_{i=1}^n
    \left(\frac{\sigma_1}{\sigma_d}(\gamma_i)\right)^{2 - \epsilon}.
  \]
  This gives us the desired estimate.
\end{proof}

Finally we turn to the proof of \Cref{cor:qi_free_products}. First, we
reproduce an argument from \cite{dt2022anosov} to prove the following:
\begin{lemma}[{See \cite[Prop. 4.1]{dt2022anosov}}]
  Let $\Gamma_1, \Gamma_2$ be finitely-generated groups, each
  containing an element with infinite order, and fix finite-index
  subgroups $\Gamma_i' < \Gamma_i$ for $i = 1,2$. Then a finite-index subgroup of the free product
  $\Gamma_1 \ast \Gamma_2$ is isomorphic to an undistorted subgroup of
  the free product $\Gamma_1' \ast \Gamma_2'$.\end{lemma}

\begin{proof}
  We may replace each $\Gamma_i'$ with a further finite-index subgroup
  which is normal in $\Gamma_i$. Now let $\Gamma_0 $ be the
  intersection of the kernels of the compositions
  $\Gamma_1 \ast \Gamma_2 \to \Gamma_i \to \Gamma_i/\Gamma_i'$. This
  intersection has finite index in $\Gamma_1 \ast \Gamma_2$. Moreover,
  it is isomorphic to a free product of the form
  \begin{equation}
    \label{eq:free_finite_index_free_product}
    \Gamma_1' \ast \cdots \ast \Gamma_1' \ast \Gamma_2' \ast \cdots
    \ast \Gamma_2' \ast \Z \ast \cdots \ast \Z.
  \end{equation}
  To see this, let $X$ be a graph of spaces for the free product
  $\Gamma_1 \ast \Gamma_2$, consisting of a pair of spaces $X_1, X_2$
  with $\pi_1X_i \simeq \Gamma_i$, attached along an edge. Let
  $X_1', X_2'$ be the covers of $X_1, X_2$ corresponding to the
  subgroups $\Gamma_1', \Gamma_2'$ respectively. Observe that the
  finite cover $X_0$ of $X$ corresponding to $\Gamma_0$ is itself a
  finite graph of spaces, and that each vertex space of $X_0$ is
  homeomorphic to one of $X_1', X_2'$; it follows that $\pi_1X_0$ is
  isomorphic to a free product of the form in
  \eqref{eq:free_finite_index_free_product}.

  However, any free product as in
  \eqref{eq:free_finite_index_free_product} itself embeds as an
  undistorted subgroup of the free product $\Gamma_1' \ast
  \Gamma_2'$. To see this, let $\gamma_i \in \Gamma_i'$ have infinite
  order; then the subgroup of $\Gamma_1' \ast \Gamma_2'$ given by
  \[
    \big \langle \gamma_2\Gamma_1'\gamma_2^{-1}, \gamma_2^2 \Gamma_1'
    \gamma_2^{-2}, \ldots \gamma_2^r\Gamma_1'\gamma_2^{-r},
    \gamma_1\Gamma_2'\gamma_1^{-1}, \gamma_1^s\Gamma_2'\gamma_1^{-s},
    \gamma_1^{s+1}\gamma_2\gamma_1^{-s-1}, \ldots, \gamma_1^{s+q}
    \gamma_2 \gamma_1^{-s-q}\big \rangle
  \]
  is naturally isomorphic to the free product of $r$ copies of
  $\Gamma_1'$, $s$ copies of $\Gamma_2'$, and $q$ copies of $\Z$.
\end{proof}

\begin{proof}[Proof of \Cref{cor:qi_free_products}]
  Suppose that $\Gamma_1, \Gamma_2$ are finitely-generated discrete
  subgroups of $\SL_d(\Real)$, each of which is quasi-isometrically
  embedded. \Cref{discrete-freeprod-1} implies that, for finite-index
  subgroups $H_1 < \Gamma_1, H_2 < \Gamma_2$, we may realize the free
  product $H_1 \ast H_2$ as a quasi-isometrically embedded subgroup of
  $\SL_n(\Real)$. Then, by the previous lemma, a finite-index subgroup
  $\Gamma_0$ of the free product $\Gamma_1 \ast \Gamma_2$ is
  quasi-isometrically embedded in $\SL_n(\Real)$.

  Now, if $\Gamma < \Gamma'$ is a subgroup of index $m$, and
  $\rho:\Gamma \to \SL_n(\Real)$ is a faithful representation, then
  one may always construct a faithful representation
  $\rho':\Gamma' \to \SL_{nm}(\Real)$ so that the restriction
  $\rho'|_{\Gamma}$ has a subrepresentation isomorphic to $\rho$. In
  particular, if $\rho$ is a quasi-isometric embedding, then so is
  $\rho'$. Applying this to $\Gamma_0 < \Gamma_1 \ast \Gamma_2$
  completes the proof.
\end{proof}

\bibliographystyle{siam}

\bibliography{biblio.bib}

\end{document}